\documentclass[draft]{birkjour}
\usepackage[noadjust]{cite}
\usepackage{xcolor}

\RequirePackage[all]{xy}

\allowdisplaybreaks

\usepackage{longtable}
\usepackage{lscape}
\usepackage{amssymb}

\newtheorem{thm}{Theorem}[section]

\newtheorem{lem}[thm]{Lemma}

\theoremstyle{definition}

\theoremstyle{remark}
\newtheorem{rem}[thm]{Remark}

\newtheorem*{ex}{Example}
\numberwithin{equation}{section}

\newcommand{\BibTeX}{B\kern-0.1emi\kern-0.017emb\kern-0.15em\TeX}
\newcommand{\XYpic}{$\mathrm{X\kern-0.3em\raisebox{-0.18em}{Y}}$-$\mathrm{pic}\,$}

\newcommand{\BR}{\mathbb{R}}
\newcommand{\BC}{\mathbb{C}}

\newcommand{\C}{C \kern -0.1em \ell} 

\def\A{{\rm A}}
\def\B{{\rm B}}
\def\Q{{\rm Q}}
\def\Z{{\rm Z}}
\def\H{{\rm H}}
\def\D{{\rm D}}
\def\ad{{\rm ad}}
\def\mod{{\rm \;mod\; }}
\def\Lambd{{\rm\Lambda}}
\def\Gamm{{\rm\Gamma}}
\def\Aut{{\rm Aut}}
\def\ker{{\rm ker}}
\def\rad{{\rm rad}\;}
\def\Pin{{\rm Pin}}
\def\Spin{{\rm Spin}}
\newcommand{\F}{\mathbb{F}}

\newcommand{\ed}{\end{document}}

\begin{document}

%
%
%
%
%
%
%
%
%

\title[Generalized Degenerate Clifford and Lipschitz Groups]
 {Generalized Degenerate Clifford and Lipschitz Groups in Geometric Algebras}
\author[E. Filimoshina]{Ekaterina Filimoshina}
\address{%
HSE University\\
Moscow 101000\\
Russia
\medskip}
\address{
and
\medskip}
\address{
Skolkovo Institute of Science and Technology \\
 Moscow 121205\\
 Russia 
}
\email{filimoshinaek@gmail.com}
%
\author[D. Shirokov]{Dmitry Shirokov}
\address{%
HSE University\\
Moscow 101000\\
Russia
\medskip}
\address{
and
\medskip}
\address{
Institute for Information Transmission Problems of the Russian Academy of Sciences \\
Moscow 127051 \\
Russia}
\email{dm.shirokov@gmail.com}
\subjclass{15A66, 11E88}
\keywords{geometric algebra, Clifford algebra, degenerate geometric algebra, Clifford group, Lipschitz group}
\date{\today}
\dedicatory{Last Revised:\\ \today}
\begin{abstract}
This paper introduces and studies generalized degenerate Clifford and Lipschitz groups in geometric (Clifford) algebras. These Lie groups preserve the direct sums of the subspaces determined by the grade involution and reversion under the adjoint and twisted adjoint representations in degenerate geometric algebras. We prove that the generalized degenerate Clifford and Lipschitz groups can be defined using centralizers and twisted centralizers of fixed grades subspaces and the norm functions that are widely used in the theory of spin groups. We study the relations between these groups and consider them in the particular cases of plane-based geometric algebras and Grassmann algebras. The corresponding Lie algebras are studied. The presented groups are interesting for the study of generalized degenerate spin groups and applications in computer science, physics, and engineering.
\end{abstract}
\label{page:firstblob}
\maketitle

\section{Introduction}

This paper introduces and studies generalized degenerate Clifford and Lipschitz groups in geometric (Clifford) algebras  $\C_{p,q,r}$ \cite{hestenes,lounesto,p}. Ordinary Clifford and Lipschitz groups preserve the grade-$1$ subspace under the adjoint and twisted adjoint representations respectively. These groups are usually considered in the non-degenerate algebras $\C_{p,q,0}$  \cite{lg1,lounesto,p}, as well as in the degenerate ones \cite{Abl,crum_book,br2,br1,dereli}, and are used in the theory of spin groups and in different applications in physics \cite{cv2,phys,br2} and computer science, in particular in equivariant neural networks \cite{cNN,cNN0,gatr,zhdanov}. The generalized degenerate Clifford and Lipschitz groups preserve the direct sums of the subspaces determined by the grade involution and reversion under the same representations as the ordinary Clifford and Lipschitz groups in the degenerate and non-degenerate $\C_{p,q,r}$. The introduced groups contain the ordinary Clifford and Lipschitz groups as subgroups and are useful for the study of generalized degenerate spin groups.

In this work, we prove that the generalized degenerate Clifford and Lipschitz groups can be defined using centralizers and twisted centralizers \cite{cen} of the subspaces of fixed grades and the norm functions that are used in the definitions of spin groups. We prove that the generalized Clifford and Lipschitz groups are related to each other through multiplication. We provide several examples of the groups in the particular cases of  the Grassmann algebras $\C_{0,0,n}$ and the Clifford algebras $\C_{p,q,1}$, including the plane-based geometric algebras (PGA) $\C_{p,0,1}$ \cite{pga}, which are widely used in applications. We find the corresponding Lie algebras and their dimensions. The current work generalizes the results of the papers \cite{GenSpin,OnInner}, where similar questions are addressed in the case of the non-degenerate algebra $\C_{p,q,0}$. The papers \cite{ICACGA,OnSomeLie} discuss the groups preserving the subspaces of fixed parity under the adjoint representation and the twisted adjoint representations in degenerate and non-degenerate $\C_{p,q,r}$. 

In Section \ref{section_def}, we discuss degenerate and non-degenerate geometric algebras $\C_{p,q,r}$ and introduce the necessary notation. Section \ref{section_centralizers} considers centralizers and twisted centralizers in $\C_{p,q,r}$ and the norm functions. In Section~\ref{section_ds_qt}, we introduce and study the generalized degenerate Clifford and Lipschitz groups. A comprehensive list of these groups is presented in Table~\ref{table_all_groups}. In Section \ref{section_relations}, we find the relations between the considered groups. Section \ref{section_examples} provides several examples of the groups in the cases of the algebras $\C_{p,q,1}$ and $\C_{0,0,n}$. In Appendix~\ref{section_algebras}, we study the corresponding Lie algebras, which are summarized in Tables~\ref{table_lie_algebras_A}--\ref{table_lie_algebras_Q2}.

This paper is an extended version of the short note (12 pages) in Conference Proceedings \cite{AB_CGI} (Empowering Novel Geometric Algebra for Graphics \& Engineering, Computer Graphics International, Geneva, Switzerland, 2024). Sections \ref{section_relations} and \ref{section_examples} and Appendix \ref{section_algebras} are new, Subsections \ref{section_checkAB} and \ref{section_tildeAB} are extended. The new results include: a proof of the relations between the generalized Lipschitz groups via multiplication, examples of these groups and their relations in particular cases $\C_{p,q,1}$ and $\C_{0,0,n}$, and a study of the corresponding Lie algebras. Theorems \ref{lemma_AH}, \ref{theorem_rel_A}, \ref{theorem_rel_B}, \ref{thm_alg} and Lemma \ref{lemma_mult_c3} are presented for the first time.

\section{The subspaces defined by the grade involution and reversion}\label{section_def}

Let us consider the geometric (Clifford) algebra \cite{hestenes,lounesto,p} $\C(V)=\C_{p,q,r}$, $p+q+r=n\geq1$, over a vector space $V$ with a symmetric bilinear form, which can be real $\BR^{p,q,r}$ or complex  $\BC^{p+q,0,r}$. 
In this work, we consider both the case of the non-degenerate geometric algebras $\C_{p,q,0}$ and the case of the degenerate geometric algebras $\C_{p,q,r}$, $r\neq0$. 
We use $\F$ to denote the field of real numbers $\BR$ in the first case and the
field of complex numbers $\BC$ in the second case respectively.
We use $\Lambd_r$ to denote the subalgebra $\C_{0,0,r}$, which is the Grassmann (exterior) algebra \cite{phys,lounesto}.
If $r=0$, then $\Lambd_0$ is a one-dimensional space spanned by the identity element.
The identity element is denoted by $e$, the generators are denoted by $e_a$, $a=1,\ldots,n$. The generators satisfy the following conditions: 
\begin{eqnarray}\label{gen}
    e_a e_b + e_b e_a = 2 \eta_{ab}e,\qquad \forall a,b=1,\ldots,n,
\end{eqnarray}
where $\eta=(\eta_{ab})$ is the diagonal matrix with $p$ times $+1$, $q$ times $-1$, and $r$ times $0$ on the diagonal in the real case $\C(\BR^{p,q,r})$ and $p+q$ times $+1$ and $r$ times $0$ on the diagonal in the complex case $\C(\BC^{p+q,0,r})$.

Consider the subspaces $\C^k_{p,q,r}$ of fixed grades $k=0,\ldots,n$. Their elements are linear combinations of the basis elements $e_A=e_{a_1\ldots a_k}:=e_{a_1}\cdots e_{a_k}$, $a_1<\cdots<a_k$, labeled by ordered multi-indices $A$ of length $k$, where $0\leq k\leq n$. The grade-$0$ subspace is denoted by $\C^0$ without the lower indices $p,q,r$, since it does not depend on the Clifford algebra's signature. 
We have $\C^k_{p,q,r}=\{0\}$ for $k<0$ and $k>n$. We use the following notation:
\begin{eqnarray}
\C^{\geq k}_{p,q,r} &:=& \C^k_{p,q,r}\oplus\C^{k+1}_{p,q,r}\oplus\cdots\oplus\C^n_{p,q,r},\label{def_geq_k}
\\
\C^{\leq k}_{p,q,r} &:=& \C^0\oplus\C^1_{p,q,r}\oplus\cdots\oplus \C^k_{p,q,r}.\label{def_leq_k}
\end{eqnarray}

The grade involution of an element  $U\in\C_{p,q,r}$ is denoted by $\widehat{U}$. The grade involution defines the even $\C^{(0)}_{p,q,r}$ and odd $\C^{(1)}_{p,q,r}$ subspaces:
\begin{eqnarray}\label{evenodd}
\C^{(k)}_{p,q,r} = \{U\in\C_{p,q,r}:\;\; \widehat{U}=(-1)^kU\}=\bigoplus_{j=k\mod{2}}\;\C^j_{p,q,r},\quad k=0,1.
\end{eqnarray}
Any element $U\in\C_{p,q,r}$ can be represented as a sum 
$U=\langle U\rangle_{(0)}+\langle U\rangle_{(1)}$, where  $\langle U\rangle_{(0)}\in\C^{(0)}_{p,q,r}$ and $\langle U\rangle_{(1)}\in\C^{(1)}_{p,q,r}$.
We use angle brackets $\langle \cdot \rangle_{(l)}$ to denote the operation of projection of multivectors onto the subspaces $\C^{(l)}_{p,q,r}$, $l=0,1$.
For an arbitrary subset $\H\subseteq\C_{p,q,r}$, we have

\begin{eqnarray}\label{proj_def}
\langle \H \rangle_{(0)} := \H\cap\C^{(0)}_{p,q,r},\qquad \langle \H \rangle_{(1)} := \H\cap\C^{(1)}_{p,q,r}.
\end{eqnarray}

The reversion is denoted by $\widetilde{U}$, the Clifford conjugation is denoted by $\widehat{\widetilde{U}}$. 
The grade involution and reversion define four subspaces $\C^{\overline{0}}_{p,q,r}$,  $\C^{\overline{1}}_{p,q,r}$,  $\C^{\overline{2}}_{p,q,r}$, and  $\C^{\overline{3}}_{p,q,r}$ (they are called the subspaces of quaternion types $0, 1, 2$, and $3$ respectively in the papers \cite{quat1,quat2,quat3}):
\begin{eqnarray}
\C^{\overline{k}}_{p,q,r}=\{U\in\C_{p,q,r}:\; \widehat{U}=(-1)^k U,\;\; \widetilde{U}=(-1)^{\frac{k(k-1)}{2}} U\},\quad k=0, 1, 2, 3.\label{qtdef}
\end{eqnarray}
Note that the Clifford algebra $\C_{p,q,r}$ can be represented as a direct sum of the subspaces $\C^{\overline{k}}_{p,q,r}$, $k=0, 1, 2, 3$, and viewed as $\mathbb{Z}_2\times\mathbb{Z}_2$-graded algebra with respect to the commutator and anticommutator \cite{b_lect}.
To denote the direct sum of different subspaces, we use the upper multi-index and omit the direct sum sign. For instance, $\C^{(1)\overline{2}4}_{p,q,r}:=\C^{(1)}_{p,q,r}\oplus\C^{\overline 2}_{p,q,r}\oplus\C^{4}_{p,q,r}$.

\section{Centralizers, twisted centralizers, and norm functions in $\C_{p,q,r}$}\label{section_centralizers}

In this section, we provide several facts about centralizers and twisted centralizers of fixed subspaces in $\C_{p,q,r}$. We use these statements to prove Theorems \ref{theorem_AB}--\ref{theorem_tildeQ} in Section \ref{section_ds_qt}. Lemmas \ref{cases_we_use} and \ref{centralizers_qt_ds} are provided in \cite{cen}, Lemma \ref{lemma_for_AB} is new.

Consider the centralizers  $\Z^{m}_{p,q,r}$ and twisted  centralizers $\check{\Z}^m_{p,q,r}$ of the fixed grade subspaces $\C^{m}_{p,q,r}$, $m=0,\ldots,n$:
\footnote{Note that we can also consider another twisted centralizer defined as $$\tilde{\Z}^m_{p,q,r}:=\{X\in\C_{p,q,r}:\quad X\langle V\rangle_{(0)} + \widehat{X} \langle V\rangle_{(1)} = VX,\quad \forall V\in\C^{m}_{p,q,r}\}.$$ It corresponds to $\tilde{\ad}$ (see the formula (\ref{twa22}) in Section \ref{section_ds_qt}). We have $\tilde{\Z}^m_{p,q,r} = \Z^m_{p,q,r}$ if $m$ is even and $\tilde{\Z}^m_{p,q,r}=\check{\Z}^m_{p,q,r}$ if $m$ is odd.}
\begin{eqnarray}
\Z^{m}_{p,q,r}&:=&\{X\in\C_{p,q,r}:\quad X V = V X,\quad \forall V\in\C^{m}_{p,q,r}\},\label{def_Zm}
\\
    \check{\Z}^m_{p,q,r}&:=&\{X\in\C_{p,q,r}:\quad \widehat{X} V = V X,\quad \forall V\in\C^{m}_{p,q,r}\}.\label{def_chZm}
\end{eqnarray}
The center of the geometric algebra $\C_{p,q,r}$ denoted by $\Z_{p,q,r}$ is  the centralizer of the grade-$1$ subspace $\C^{1}_{p,q,r}$ (see the formula (\ref{cc_2}) in Lemma \ref{cases_we_use}) and of the entire geometric algebra $\C_{p,q,r}$ as well.
It is well known (see, for example, \cite{br1}) that
\begin{eqnarray}\label{Zpqr}
\Z_{p,q,r}=
\left\lbrace
\begin{array}{lll}
\Lambd^{(0)}_{r}\oplus\C^n_{p,q,r},&&\mbox{$n$ is odd},
\\
\Lambd^{(0)}_{r},&&\mbox{$n$ is even}.
\end{array}
\right.
\end{eqnarray}
The paper \cite{cen} finds an explicit form of the centralizers $\Z^m_{p,q,r}$ and twisted centralizers $\check{\Z}^m_{p,q,r}$ for any $m=0,\ldots,n$ (Theorem 3.6). 
In Lemma \ref{cases_we_use}, for the reader's convenience, we write out the particular cases that we use in Section \ref{section_ds_qt} of this paper.
A few words on the notation in Lemma \ref{cases_we_use}. The spaces $\C^k_{p,q,0}$ and $\Lambd^k_r$, $k=0,\ldots,n$, are regarded as subspaces of $\C_{p,q,r}$. 
If $k>r$ or $k<0$, then $\Lambd^k_r=\{0\}$.
By $\{\C^k_{p,q,0}\Lambd^l_r\}$, we denote the subspace of $\C_{p,q,r}$ spanned by elements of the form $ab$, where $a\in \C^k_{p,q,0}$ and $b\in\Lambd^l_r$.

\begin{lem}[\cite{cen}]\label{cases_we_use}
The centralizers and twisted centralizers of the subspaces of fixed grades have the following form:
\begin{eqnarray}
\!\!\!\!\!\!\!\!\!\!\!\!\!&&\Z^1_{p,q,r}=\Z_{p,q,r};\qquad \Z^2_{p,q,r}=
\left\lbrace
    \begin{array}{lll}
\Lambd_r\oplus\C^{n}_{p,q,r},&& r\neq n,
\\
\Lambd_r,&& r=n;
\end{array}
    \right.\label{cc_2}
\\
\!\!\!\!\!\!\!\!\!\!\!\!\!\!\!\!&&\Z^3_{p,q,r}\!=\!\left\lbrace
            \begin{array}{lll}\label{cc_3}
            \Lambd^{(0)}_r\oplus\Lambd^{n-2}_r\oplus \{\C^{1}_{p,q,0}(\Lambd^{n-3}_r\oplus\Lambd^{n-2}_r)\}
            \\
            \quad\oplus\{\C^{2}_{p,q,0}\Lambd^{n-3}_r\}\!\oplus\!\C^{n}_{p,q,r}, \!\!\!\!\!&&\mbox{$n$ is odd},
            \\
            \Lambd^{(0)}_r\!\oplus\!\Lambd^{n-1}_r\oplus \{\C^{1}_{p,q,0}\Lambd^{\geq n-2}_r\}\oplus \{\C^{2}_{p,q,0}\Lambd^{n-2}_r\}, \!\!\!\!\!&&\mbox{$n$ is even};
            \end{array}
            \right.
\\
\!\!\!\!\!\!\!\!\!\!\!\!\!\!\!\!\!\!\!\!\!\!&&\Z^4_{p,q,r}\!=\!
\left\lbrace
\begin{array}{lll}\label{cc_4}
\Lambd_r\oplus \{\C^{1}_{p,q,0}(\Lambd^{n-3}_r\oplus\Lambd^{n-2}_r)\}
\\
\quad\oplus\{\C^{2}_{p,q,0}(\Lambd^{n-4}_r\oplus\Lambd^{n-3}_r)\}\oplus\C^{n}_{p,q,r}, && r\neq n,
\\
\Lambd_r, && r=n;
\end{array}
\right.
\\
   \!\!\!\! \!\!\!\!\!\!\!\!\!\!\!\!\!\!\!\!\!\!&&\check{\Z}^1_{p,q,r}=\Lambd_r;\qquad \check{\Z}^3_{p,q,r}=\Lambd_r\oplus \{\C^{1}_{p,q,0}\Lambd^{\geq n-2}_r\}\oplus \{\C^{2}_{p,q,0}\Lambd^{\geq n-3}_r\};\label{ch_cc_1}
    \\
    \!\!\!\!\!\!\!\!\!\!\!\!\!\!\!\!\!\!\!\!\!\!&&\check{\Z}^2_{p,q,r}=
    \left\lbrace
    \begin{array}{lll}\label{ch_cc_2}
    \Lambd^{(0)}_r\oplus\Lambd^{n}_r\oplus \{\C^{1}_{p,q,0}\Lambd^{n-1}_r\},\!\!\!\!\!\!\!\!\!\!&&\mbox{$n$ is odd},
    \\
    \Lambd^{(0)}_r\oplus\Lambd^{n-1}_r\oplus \{\C^{1}_{p,q,0}\Lambd^{n-2}_r\}\oplus\C^{n}_{p,q,r},\!\!\!\!\!\!\!\!\!\!&&\mbox{$n$ is even},\;\; r\neq n,
    \\
    \Lambd^{(0)}_r\oplus\Lambd^{n-1}_r,\!\!\!\!\!\!\!\!\!\!&&\mbox{$n$ is even},\;\; r=n.
    \end{array}
    \right.
\end{eqnarray}
\end{lem}

In this work, we also use the centralizers $\Z^{\overline{k}}_{p,q,r}$ and twisted  centralizers $\check{\Z}^{\overline{k}}_{p,q,r}$ of  the subspaces $\C^{\overline{k}}_{p,q,r}$~(\ref{qtdef}), $k=0,1,2,3$. They are defined as:
\begin{eqnarray}
\Z^{\overline{k}}_{p,q,r}&:=&\{X\in\C_{p,q,r}:\quad X V = V X,\quad \forall V\in\C^{\overline{k}}_{p,q,r}\},\label{def_CC_ov}
\\
\check{\Z}^{\overline{k}}_{p,q,r}&:=&\{X\in\C_{p,q,r}:\quad \widehat{X} V = V X,\quad \forall V\in\C^{\overline{k}}_{p,q,r}\}.\label{def_chCC_ov}
\end{eqnarray}
The centralizers $\Z^{\overline{kl}}_{p,q,r}$ and twisted centralizers $\check{\Z}^{\overline{kl}}_{p,q,r}$ of the direct sums $\C^{\overline{kl}}_{p,q,r}=\C^{\overline{k}}_{p,q,r}\oplus\C^{\overline{l}}_{p,q,r}$, $k,l=0,1,2,3$, are defined as
\begin{eqnarray}\label{def_sumZ}
\Z^{\overline{kl}}_{p,q,r}:=\Z^{\overline{k}}_{p,q,r}\cap\Z^{\overline{l}}_{p,q,r},\qquad\check{\Z}^{\overline{kl}}_{p,q,r}:=\check{\Z}^{\overline{k}}_{p,q,r}\cap\check{\Z}^{\overline{l}}_{p,q,r}.
\end{eqnarray}

\begin{lem}[\cite{cen}]\label{centralizers_qt_ds}
The centralizers and twisted centralizers of the subspaces $\C^{\overline{m}}_{p,q,r}$ have the form
\begin{eqnarray*}
&\Z^{\overline{m}}_{p,q,r}=\Z^{m}_{p,q,r},\qquad\check{\Z}^{\overline{m}}_{p,q,r}=\check{\Z}^{m}_{p,q,r},\qquad m=1,2,3;
\\
&\Z^{\overline{0}}_{p,q,r}=\Z^{4}_{p,q,r},\qquad \check{\Z}^{\overline{0}}_{p,q,r}=\langle\Z^{4}_{p,q,r}\rangle_{(0)}.\label{f_f3}
\end{eqnarray*}
The centralizers and twisted centralizers of the direct sums $\C^{\overline{kl}}_{p,q,r}$ are
    \begin{eqnarray}
    &\!\!\!\!\!\!\Z^{\overline{01}}_{p,q,r}=\Z^{\overline{12}}_{p,q,r}=\Z^{\overline{13}}_{p,q,r}=\Z_{p,q,r},\quad \Z^{\overline{23}}_{p,q,r}=\Z^2_{p,q,r}\cap\Z^3_{p,q,r},\label{centralizers_qt_ds_1}
    \\
    &\Z^{\overline{02}}_{p,q,r}=\Z^2_{p,q,r},\quad \Z^{\overline{03}}_{p,q,r}=\Z^3_{p,q,r},\label{centralizers_qt_ds_2_00}
    \\
    &\!\!\!\!\!\! \check{\Z}^{\overline{12}}_{p,q,r}=\check{\Z}^1_{p,q,r}\cap\check{\Z}^2_{p,q,r},\quad \check{\Z}^{\overline{23}}_{p,q,r}=\check{\Z}^2_{p,q,r}\cap\check{\Z}^3_{p,q,r},\quad \check{\Z}^{\overline{13}}_{p,q,r}=\check{\Z}^1_{p,q,r}, \label{centralizers_qt_ds_2_0}
    \\
    &\check{\Z}^{\overline{01}}_{p,q,r}=\langle{\Z}^1_{p,q,r}\rangle_{(0)},\quad \check{\Z}^{\overline{02}}_{p,q,r}=\langle{\Z}^2_{p,q,r}\rangle_{(0)},\quad\check{\Z}^{\overline{03}}_{p,q,r}=\langle\Z^3_{p,q,r}\rangle_{(0)}.\label{centralizers_qt_ds_2}
\end{eqnarray}
\end{lem}

We use the upper index $\times$ to denote the subset $\H^{\times}$ of all invertible elements of any set $\H$. 
Consider the well-known (see, for example, \cite{Abl,ABS,lg1,lounesto}) Clifford and Lipschitz groups, which are defined in the following way respectively:
\begin{eqnarray}
    \Gamm_{p,q,r}&:=&\{T\in\C^{\times}_{p,q,r}:\quad T\C^{1}_{p,q,r}T^{-1}\subseteq\C^{1}_{p,q,r}\},\label{def_cg}
    \\
    \Gamm^{\pm}_{p,q,r}&:=&\{T\in\C^{\times}_{p,q,r}:\quad \widehat{T}\C^{1}_{p,q,r}T^{-1}\subseteq\C^{1}_{p,q,r}\}.\label{def_lg}
\end{eqnarray}
We use the upper index $\pm$ in the notation $\Gamm^{\pm}_{p,q,r}$ to align with its common notation $\Gamm^{\pm}_{p,q,0}$ in the literature (see, for example, \cite{lg1}) in the non-degenerate geometric algebras $\C_{p,q,0}$. In the case $\C_{p,q,0}$, this index indicates that  $\Gamm^{\pm}_{p,q,0}$ consists of even and odd elements and has the following equivalent definition:
\begin{eqnarray}
    \Gamm^{\pm}_{p,q,0}= \{T\in\C^{(0)\times}_{p,q,0}\cup\C^{(1)\times}_{p,q,0}:\quad {T}\C^{1}_{p,q,0}T^{-1}\subseteq\C^{1}_{p,q,0}\}.\label{def_lg_2}
\end{eqnarray}
Consider two norm functions widely used in the theory of spin groups \cite{ABS,lg1,lounesto}:
\begin{eqnarray}\label{norm_functions}
    \psi(T):=\widetilde{T}T,\qquad \chi(T):=\widehat{\widetilde{T}}T,\qquad \forall T\in\C_{p,q,r}.
\end{eqnarray}
For example, in the case of the non-degenerate geometric algebra $\C_{p,q,0}$, the groups  $\Pin_{p,q,0}$ and $\Spin_{p,q,0}$ are defined as \cite{lounesto,p}:
\begin{eqnarray*}
    \Pin_{p,q,0}&:=&\{T\in \Gamm^{\pm}_{p,q,0}:\; \widetilde{T}T=\pm e\}=\{T\in \Gamm^{\pm}_{p,q,0}:\; \widehat{\widetilde{T}}T=\pm e\},\label{pin_ex}
    \\
    \Spin_{p,q,0}&:=&\{T\in \langle\Gamm^{\pm}_{p,q,0}\rangle_{(0)}:\;\widetilde{T}T=\pm e\}=\{T\in \langle\Gamm^{\pm}_{p,q,0}\rangle_{(0)}:\;\widehat{\widetilde{T}}T=\pm e\}.\label{spin_ex}
\end{eqnarray*}
 Note that
\begin{eqnarray}\label{norms_01_03}
    \widetilde{T}T\in\C^{\overline{01}}_{p,q,r},\qquad \widehat{\widetilde{T}}T\in\C^{\overline{03}}_{p,q,r},\qquad \forall T\in\C_{p,q,r}.
\end{eqnarray}
The proof of (\ref{norms_01_03}) in the case of the degenerate geometric algebra $\C_{p,q,r}$ repeats the proof of this statement in the special case of the non-degenerate algebra $\C_{p,q,0}$ given in \cite{OnInner} (see Lemmas 4.1 and 5.1). 

Let us prove Lemma \ref{lemma_for_AB} about such $T\in\C_{p,q,r}$ that $\psi(T)$ and $\chi(T)$
are in the centralizers $\Z^{\overline{m}}_{p,q,r}$ (\ref{def_CC_ov}) and twisted centralizers $\check{\Z}^{\overline{m}}_{p,q,r}$ (\ref{def_chCC_ov}). 
The proofs of Theorems \ref{theorem_AB}--\ref{theorem_tildeQ} in Section \ref{section_ds_qt} are based on this lemma.

\begin{lem}\label{lemma_for_AB}
For any $T\in\C^{\times}_{p,q,r}$, in the cases $(k,l)=(0,1),(1,0),(2,3),(3,2)$, we have:
\begin{eqnarray}
\!\!\!\!\!\!\!\!T\C^{\overline{k}}_{p,q,r} T^{-1}\subseteq \C^{\overline{kl}}_{p,q,r}\quad \Leftrightarrow\quad\widetilde{T}T\in\Z^{\overline{k}\times}_{p,q,r},\label{1_for_AB_0}
\\
 \widehat{T} \C^{\overline{k}}_{p,q,r} T^{-1}\subseteq \C^{\overline{kl}}_{p,q,r}\quad \Leftrightarrow\quad\widehat{\widetilde{T}}T\in\check{\Z}^{\overline{k}\times}_{p,q,r},\label{1_for_AB}
 \end{eqnarray}
and in the cases $(k,l)=(0,3),(3,0),(1,2),(2,1)$, we have:
 \begin{eqnarray}
  \!\!\!\!\!\!\!\!T\C^{\overline{k}}_{p,q,r} T^{-1}\subseteq \C^{\overline{kl}}_{p,q,r}\quad\Leftrightarrow\quad\widehat{\widetilde{T}}T\in\Z^{\overline{k}\times}_{p,q,r},\label{2_for_AB_0}
  \\
 \widehat{T} \C^{\overline{k}}_{p,q,r} T^{-1}\subseteq \C^{\overline{kl}}_{p,q,r}\quad\Leftrightarrow\quad \widetilde{T}T\in\check{\Z}^{\overline{k}\times}_{p,q,r}.\label{2_for_AB}
\end{eqnarray}
\end{lem}
\begin{proof}
Let us prove  (\ref{1_for_AB_0}) and (\ref{2_for_AB_0}). Suppose $T\in\C^{\times}_{p,q,r}$ satisfies $ T\C^{\overline{k}}_{p,q,r} T^{-1}\subseteq \C^{\overline{kl}}_{p,q,r}$. Consider  $(k,l)=(0,1),(1,0),(2,3),(3,2)$. We have
\begin{eqnarray}
T U_{\overline{k}} T^{-1} = (T \widetilde{U_{\overline{k}}} T^{-1}){\widetilde{\;\;}}=\widetilde{T^{-1}}U_{\overline{k}}\widetilde{T},\qquad \forall U_{\overline{k}}\in\C^{\overline{k}}_{p,q,r}.\label{TUT_2}
\end{eqnarray}
Multiplying both sides of the equation (\ref{TUT_2}) on the left by $\widetilde{T}$, on the right by $T$, we get $(\widetilde{T}T)U_{\overline{k}}=U_{\overline{k}}(\widetilde{T}T)$ for any $U_{\overline{k}}\in\C^{\overline{k}}_{p,q,r}$. Thus, $\widetilde{T}T\in\Z^{\overline{k}\times}_{p,q,r}$. In the cases $(k,l)=(0,3),(3,0),(1,2),(2,1)$, we obtain
\begin{eqnarray}
    T U_{\overline{k}} T^{-1}=(T \widehat{\widetilde{U_{\overline{k}}}} T^{-1}){\widehat{\widetilde{\;\;}}}=\widehat{\widetilde{T^{-1}}}U_{\overline{k}}\widehat{\widetilde{T}},\qquad \forall U_{\overline{k}}\in\C^{\overline{k}}_{p,q,r}.\label{TTU=UTT_2}
\end{eqnarray}
Multiplying both sides of the equation (\ref{TTU=UTT_2}) on the left by $\widehat{\widetilde{T}}$, on the right by $T$, we get
$(\widehat{\widetilde{T}}T)U_{\overline{k}}=U_{\overline{k}}(\widehat{\widetilde{T}}T)$ for any $U_{\overline{k}}\in\C^{\overline{k}}_{p,q,r}$; therefore, $\widehat{\widetilde{T}}T\in\Z^{\overline{k}\times}_{p,q,r}$.

Suppose $T\in\C^{\times}_{p,q,r}$ satisfies $\widetilde{T}T=V\in\Z^{\overline{k}\times}_{p,q,r}$; then $\widetilde{T}=V T^{-1}$ and $\widetilde{T^{-1}}=T V^{-1}$. For any $U_{\overline{k}}\in\C^{\overline{k}}_{p,q,r}$, 
\begin{eqnarray*}
({T} U_{\overline{k}} T^{-1}){\widetilde{\;\;}}=
{\widetilde{T^{-1}}}\widetilde{U_{\overline{k}}}\widetilde{T}=
(T V^{-1})\widetilde{U_{\overline{k}}}(V T^{-1})=T V^{-1} V  \widetilde{U_{\overline{k}}} T^{-1} = T \widetilde{U_{\overline{k}}}T^{-1}.\label{TUT_1}
\end{eqnarray*}
Therefore, ${T} U_{\overline{k}} T^{-1}\in\C^{\overline{01}}_{p,q,r}$ if $k=0,1$ and ${T} U_{\overline{k}} T^{-1}\in\C^{\overline{23}}_{p,q,r}$ if $k=2,3$, which completes the proof. Similarly, suppose $T\in\C^{\times}_{p,q,r}$ satisfies $\widehat{\widetilde{T}}T=V\in\Z^{\overline{k}\times}_{p,q,r}$; then  $\widehat{\widetilde{T}}=VT^{-1}$ and $\widehat{\widetilde{T^{-1}}}=TV^{-1}$. For any $U_{\overline{k}}\in\C^{\overline{k}}_{p,q,r}$, we have
\begin{eqnarray*}
(T U_{\overline{k}} T^{-1}){\widehat{\widetilde{\;\;}}}=\widehat{\widetilde{T^{-1}}} \widehat{\widetilde{U_{\overline{k}}}} \widehat{\widetilde{T}}=(T V^{-1})\widehat{\widetilde{U_{\overline{k}}}} (V T^{-1}) = T V^{-1} V \widehat{\widetilde{U_{\overline{k}}}} T^{-1} = T \widehat{\widetilde{U_{\overline{k}}}} T^{-1}.\label{TUT_3}
\end{eqnarray*}
Thus, $T U_{\overline{k}} T^{-1}\in\C^{\overline{03}}_{p,q,r}$ if $k=0,3$ and $T U_{\overline{k}} T^{-1}\in\C^{\overline{12}}_{p,q,r}$  if $k=1,2$, and the proof is completed.

Now we prove (\ref{1_for_AB}) and (\ref{2_for_AB}).  Suppose $T\in\C^{\times}_{p,q,r}$ satisfies $\widehat{T}\C^{\overline{k}}_{p,q,r}T^{-1}\subseteq\C^{\overline{kl}}_{p,q,r}$. For any $U_{\overline{k}}\in\C^{\overline{k}}_{p,q,r}$, we have
\begin{eqnarray}
&&\!\!\!\!\!\!\!\!\!\!\!\!\!\!\!\!\!\!\!\!\!\!\!\!\!\widehat{T}U_{\overline{k}} T^{-1}=(\widehat{T}\widetilde{U_{\overline{k}}} T^{-1})\widetilde{\;\;}=\widetilde{T^{-1}} U_{\overline{k}} \widehat{\widetilde{T}},\;\; (k,l)=(0,1),(1,0),(2,3),(3,2),\label{Tut01_}
\\
&&\!\!\!\!\!\!\!\!\!\!\!\!\!\!\!\!\!\!\!\!\!\!\!\!\!\widehat{T}U_{\overline{k}} T^{-1}=(\widehat{T}\widehat{\widetilde{U_{\overline{k}}}} T^{-1}){\widehat{\widetilde{\;\;}}}=\widehat{\widetilde{T^{-1}}} U_{\overline{k}}\widetilde{T},\;\; (k,l)=(0,3),(3,0),(1,2),(2,1).\label{TuTov12_r}
\end{eqnarray} 
Multiplying both sides of the equation  (\ref{Tut01_}) on the left by  $\widetilde{T}$, on the right by $T$,  we get $\widehat{(\widehat{\widetilde{T}}T)}U_{\overline{kl}}=U_{\overline{kl}}(\widehat{\widetilde{T}}T)$ and $\widehat{\widetilde{T}}T\in\check{\Z}^{\overline{k}\times}_{p,q,r}$.
Multiplying both sides of the equation (\ref{TuTov12_r}) on the left by $\widehat{\widetilde{T}}$, on the right by $T$, we obtain 
$\widehat{(\widetilde{T}T)}U_{\overline{k}}=U_{\overline{k}}(\widetilde{T}T)$; therefore, $\widetilde{T}T\in\check{\Z}^{\overline{k}\times}_{p,q,r}$.

Suppose $T\in\C^{\times}_{p,q,r}$ satisfies $\widehat{\widetilde{T}}T=V\in\check{\Z}^{\overline{k}\times}_{p,q,r}$. For any $U_{\overline{k}}\in\C^{\overline{k}}_{p,q,r}$,
\begin{eqnarray*}
(\widehat{T} U_{\overline{k}} T^{-1}){\widetilde{\;\;}}=
{\widetilde{T^{-1}}}\widetilde{U_{\overline{k}}}\widehat{\widetilde{T}}=\widehat{(T V^{-1})} \widetilde{U_{\overline{k}}} (V T^{-1})=\widehat{T} \widehat{V^{-1}} \widehat{V} \widetilde{U_{\overline{k}}} T^{-1}=\widehat{T} \widetilde{U_{\overline{k}}} T^{-1};\label{br_2}
\end{eqnarray*}
therefore, $\widehat{T} U_{\overline{k}} T^{-1}\in\C^{\overline{01}}_{p,q,r}$ if $k=0,1$ and $\widehat{T} U_{\overline{k}} T^{-1}\in\C^{\overline{23}}_{p,q,r}$ if $k=2,3$.
Now suppose that $T\in\C^{\times}_{p,q,r}$ satisfies $\widetilde{T}T=V\in\check{\Z}^{\overline{k}\times}_{p,q,r}$; then for any $U_{\overline{k}}\in\C^{\overline{k}}_{p,q,r}$, we have
\begin{eqnarray*}
(\widehat{T} U_{\overline{k}} T^{-1})\widehat{\widetilde{\;\;}}=
\widehat{\widetilde{T^{-1}}}\widehat{\widetilde{U_{\overline{k}}}}\widetilde{T}=
\widehat{(T V^{-1})}\widehat{\widetilde{U_{\overline{k}}}}(V T^{-1})= \widehat{T}\widehat{V^{-1}} \widehat{V} \widehat{\widetilde{U_{\overline{k}}}}T^{-1}=\widehat{T}\widehat{\widetilde{U_{\overline{k}}}}T^{-1}.\label{tut_00n__2}
\end{eqnarray*}
Thus, $\widehat{T} U_{\overline{k}} T^{-1}\in\C^{\overline{03}}_{p,q,r}$ if $k=0,3$ and $\widehat{T} U_{\overline{k}} T^{-1}\in\C^{\overline{12}}_{p,q,r}$ if $k=1,2$, and the proof is completed. 
\end{proof}

\section{The groups preserving the direct sums of the subspaces determined by the grade involution and reversion under $\ad$, $\check{\ad}$, and $\tilde{\ad}$}\label{section_ds_qt}
Let us consider the adjoint representation $\ad$ acting on the group of all invertible elements $\ad:\C^{\times}_{p,q,r}\rightarrow\Aut(\C_{p,q,r})$ as $T\mapsto\ad_T$, where $\ad_{T}:\C_{p,q,r}\rightarrow\C_{p,q,r}$:
\begin{eqnarray}\label{ar}
\ad_{T}(U):=TU T^{-1},\qquad U\in\C_{p,q,r},\qquad T\in\C^{\times}_{p,q,r}.
\end{eqnarray}
Also consider the (two different) twisted adjoint representations $\check{\ad}$ and $\tilde{\ad}$ acting on the group of all invertible elements $\check{\ad},\tilde{\ad}:\C^{\times}_{p,q,r}\rightarrow\Aut(\C_{p,q,r})$ as $T\mapsto\check{\ad}_T$, $T\mapsto\tilde{\ad}_T$ respectively with $\check{\ad}_{T},\tilde{\ad}_T:\C_{p,q,r}\rightarrow\C_{p,q,r}$:
\begin{eqnarray}
&\check{\ad}_{T}(U):=\widehat{T}U T^{-1},\quad U\in\C_{p,q,r},\quad T\in\C^{\times}_{p,q,r};\label{twa1}
\\
&\tilde{\ad}_{T}(U):=T\langle U\rangle_{(0)} T^{-1}+\widehat{T} \langle U\rangle_{(1)} T^{-1},\quad U\in\C_{p,q,r},\quad T\in\C^{\times}_{p,q,r}.\label{twa22}
\end{eqnarray}
Note that both of these operations are used in the literature (see the details and the comparison of the approaches, for example, in \cite{OnSomeLie}).
Kernels of these representations have the following form (see, for example, \cite{OnSomeLie}):
\begin{eqnarray}
\!\!\!\!\!\!\!\!\!\!\!\!\!&&\ker{(\ad)}=\{T\in\C^{\times}_{p,q,r}:\; T U T^{-1}=U,\quad \forall U \in\C_{p,q,r}\}=\Z_{p,q,r}^{\times},\label{ker_ad}
    \\
\!\!\!\!\!\!\!\!\!\!\!\!\!&&\ker(\check{\ad})=\{T\in\C^{\times}_{p,q,r}:\;\widehat{T}UT^{-1}=U,\quad \forall U\in\C_{p,q,r}\}=\Lambd^{(0)\times}_r,\label{ker_check_ad}
    \\
\!\!\!\!\!\!\!\!\!\!\!\!\!&&\ker(\tilde{\ad})=\{T\in\C^{\times}_{p,q,r}: \;T \langle U\rangle_{(0)} T^{-1} \!\!+\!\widehat{T}\langle U\rangle_{(1)}T^{-1}\!=\!U,\;\; \forall U\in\C_{p,q,r}\}\!=\!\Lambd^{\times}_r.\label{ker_tilde_ad}\nonumber
\end{eqnarray}

Let us consider setwise stabilizers (see, for example, \cite{alg}) of the  subspaces $\C^{\overline{kl}}_{p,q,r}$ (\ref{qtdef}), $k,l=0,1,2,3$, in the group $\C^{\times}_{p,q,r}$ under the group actions $\ad$, $\check{\ad}$, and $\tilde{\ad}$.
We use the following notation for the groups preserving these subspaces under the adjoint representation $\ad$ (\ref{ar}):
\begin{eqnarray}\label{def_gammakl}
\Gamm^{\overline{kl}}_{p,q,r}:=\{T\in\C^{\times}_{p,q,r}:\quad \ad_{T}(\C^{\overline{kl}}_{p,q,r}):=T\C^{\overline{kl}}_{p,q,r}T^{-1}\subseteq\C^{\overline{kl}}_{p,q,r}\}.
\end{eqnarray}
Note that the groups $\Gamm^{\overline{kl}}_{p,q,r}$ are the normalizers of  $\C^{\overline{kl}}_{p,q,r}$ in $\C^{\times}_{p,q,r}$.
The following notation is used for the stabilizers of $\C^{\overline{kl}}_{p,q,r}$, $k,l=0,1,2,3$, in $\C^{\times}_{p,q,r}$ under the twisted adjoint representations $\check{\ad}$ (\ref{twa1}) and $\tilde{\ad}$ (\ref{twa22}) respectively:
\begin{eqnarray}
\!\!\!\!\!\!\!\!\!\check{\Gamm}^{\overline{kl}}_{p,q,r}&:=&\{T\in\C^{\times}_{p,q,r}:\quad \check{\ad}_{T}(\C^{\overline{kl}}_{p,q,r}):=\widehat{T}\C^{\overline{kl}}_{p,q,r}T^{-1}\subseteq\C^{\overline{kl}}_{p,q,r}\},\label{ch_Gamma_kl}
\\
\!\!\!\!\!\!\!\!\!\tilde{\Gamm}^{\overline{kl}}_{p,q,r}&:=&\{T\in\C^{\times}_{p,q,r}:\quad \tilde{\ad}_{T}(\C^{\overline{kl}}_{p,q,r})\subseteq\C^{\overline{kl}}_{p,q,r}\}.\label{ti_Gamma_kl}
\end{eqnarray}

In Subsections \ref{sectionAB}--\ref{section_tildeAB}, we find equivalent definitions of the groups $\Gamm^{\overline{kl}}_{p,q,r}$, $\check{\Gamm}^{\overline{kl}}_{p,q,r}$, and $\tilde{\Gamm}^{\overline{kl}}_{p,q,r}$ respectively. 
Note that the groups $\Gamm^{\overline{02}}_{p,q,r}$,  $\check{\Gamm}^{\overline{02}}_{p,q,r}$,  $\tilde{\Gamm}^{\overline{02}}_{p,q,r}$, $\Gamm^{\overline{13}}_{p,q,r}$,  $\check{\Gamm}^{\overline{13}}_{p,q,r}$,  and $\tilde{\Gamm}^{\overline{13}}_{p,q,r}$  (preserving the subspaces $\C^{\overline{02}}_{p,q,r}=\C^{(0)}_{p,q,r}$ and $\C^{\overline{13}}_{p,q,r}=\C^{(1)}_{p,q,r}$ under $\ad$, $\check{\ad}$, and $\tilde{\ad}$ respectively) are considered in details in the paper \cite{OnSomeLie}, where they are denoted by $\Gamm^{(0)}_{p,q,r}$, $\check{\Gamm}^{(0)}_{p,q,r}$, $\tilde{\Gamm}^{(0)}_{p,q,r}$, $\Gamm^{(1)}_{p,q,r}$, $\check{\Gamm}^{(1)}_{p,q,r}$, and $\tilde{\Gamm}^{(1)}_{p,q,r}$ respectively.

\subsection{Groups $\A^{\overline{01}}_{p,q,r}$, $\B^{\overline{12}}_{p,q,r}$, $\A^{\overline{23}}_{p,q,r}$, $\B^{\overline{03}}_{p,q,r}$, $\Gamm^{\overline{01}}_{p,q,r}$, $\Gamm^{\overline{12}}_{p,q,r}$, $\Gamm^{\overline{23}}_{p,q,r}$, $\Gamm^{\overline{03}}_{p,q,r}$}\label{sectionAB}

Let us consider the groups $\A^{\overline{01}}_{p,q,r}$ and $\B^{\overline{12}}_{p,q,r}$ with the following definitions:
\begin{eqnarray}
\!\!\!\!\!\!\!\!\!\!\!\!\!\!\A^{\overline{01}}_{p,q,r} \!\!\!\!&:=&\!\!\!\! \{T\in\C^{\times}_{p,q,r}:\; \psi(T)=\widetilde{T}T\in\Z^{1\times}_{p,q,r}=\ker({\ad})\}=\psi^{-1}(\Z^{1\times}_{p,q,r}),\label{def_A01}
\\
\!\!\!\!\!\!\!\!\!\!\!\!\!\!\B^{\overline{12}}_{p,q,r} \!\!\!\!&:=& \!\!\!\!\{T\in\C^{\times}_{p,q,r}:\;\chi(T)=\widehat{\widetilde{T}}T\in\Z^{1\times}_{p,q,r}=\ker({\ad})\}=\chi^{-1}(\Z^{1\times}_{p,q,r}),\label{def_B12}
\end{eqnarray}
where $\ker(\ad)$ (\ref{ker_ad}) is the kernel of the adjoint representation $\ad$ (\ref{ar}). 
Also consider the groups $\A^{\overline{23}}_{p,q,r}$ and $\B^{\overline{03}}_{p,q,r}$ defined as:
\begin{eqnarray}
\A^{\overline{23}}_{p,q,r}&:=&
 \{T\in\C^{\times}_{p,q,r}: \quad\widetilde{T}T\in(\Z^{2}_{p,q,r}\cap\Z^{3}_{p,q,r})^{\times}\}\label{def_A23}
 \\
 &=&\!\psi^{-1}((\Z^{2}_{p,q,r}\cap\Z^{3}_{p,q,r})^{\times}),
\\
\B^{\overline{03}}_{p,q,r}&:=&\{T\in\C^{\times}_{p,q,r}: \quad\widehat{\widetilde{T}}T\in\Z^{3\times}_{p,q,r}\}=\chi^{-1}(\Z^{3\times}_{p,q,r}),\label{def_B03}
\end{eqnarray}
where $\Z^{2}_{p,q,r}$ and $\Z^{3}_{p,q,r}$ are the sets of the elements commuting with all the elements of the subspaces  $\C^{2}_{p,q,r}$ and $\C^{3}_{p,q,r}$  respectively (see Lemma \ref{cases_we_use}) and
\begin{eqnarray}\label{cc2_cap_cc3}
    \Z^{2}_{p,q,r}\cap\Z^{3}_{p,q,r}=
     \left\lbrace
    \begin{array}{lll}
    \Lambd^{(0)}_r\oplus \Lambd^{n-2}_r\oplus\C^{n}_{p,q,r},&&\mbox{$n$ is odd},
    \\
    \Lambd^{(0)}_r\oplus \Lambd^{n-1}_r,&&\mbox{$n$ is even}.
    \end{array}
    \right.
\end{eqnarray}
The groups $\A^{\overline{01}}_{p,q,r}$, $\A^{\overline{23}}_{p,q,r}$, $\B^{\overline{12}}_{p,q,r}$, and $\B^{\overline{03}}_{p,q,r}$ are generalizations of the groups $\A$ and $\B$ \cite{OnInner} respectively in the non-degenerate geometric algebras $\C_{p,q,0}$  to the case of the degenerate geometric algebras $\C_{p,q,r}$ and coincide with them if $r=0$:
    \begin{eqnarray}
        &&\A^{\overline{01}}_{p,q,0}=\A^{\overline{23}}_{p,q,0}=\A=\{T\in\C^{\times}_{p,q,0}:\quad \widetilde{T}T\in\Z^{\times}_{p,q,0}\},\label{AB_pq0_1}
        \\
        &&\B^{\overline{12}}_{p,q,0}=\B^{\overline{03}}_{p,q,0}=\B=\{T\in\C^{\times}_{p,q,0}:\quad \widehat{\widetilde{T}}T\in\Z^{\times}_{p,q,0}\}.\label{AB_pq0_2}
    \end{eqnarray}

\begin{thm}\label{theorem_AB}
In the degenerate and non-degenerate geometric algebras $\C_{p,q,r}$,
\begin{eqnarray}
\!\!\!\!\!\!\!\!\!\!\!\!\!\!\A^{\overline{01}}_{p,q,r}=\psi^{-1}(\Z^{1\times}_{p,q,r})=\Gamm^{\overline{01}}_{p,q,r}\!\!\!\!\!&\subseteq&\!\!\!\!\!\A^{\overline{23}}_{p,q,r}\!=\psi^{-1}((\Z^{2}_{p,q,r}\cap\Z^{3}_{p,q,r})^{\times})\!=\Gamm^{\overline{23}}_{p,q,r},\label{theorem_AB_1}
\\
\!\!\!\!\!\!\!\!\!\!\!\!\!\!\B^{\overline{12}}_{p,q,r}=\chi^{-1}(\Z^{1\times}_{p,q,r})=\Gamm^{\overline{12}}_{p,q,r}\!\!\!\!\!&\subseteq& \!\!\!\!\!\B^{\overline{03}}_{p,q,r}=\chi^{-1}(\Z^{3\times}_{p,q,r})=\Gamm^{\overline{03}}_{p,q,r}.\label{theorem_AB_2}
\end{eqnarray}
\end{thm}
\begin{proof}
The inclusions $\A^{\overline{01}}_{p,q,r}\subseteq\A^{\overline{23}}_{p,q,r}$ and $\B^{\overline{12}}_{p,q,r}\subseteq\B^{\overline{03}}_{p,q,r}$ follow from the definitions (\ref{def_A01})--(\ref{def_B03}).
Let us prove the equalities in (\ref{theorem_AB_1}) and (\ref{theorem_AB_2}). We have
\begin{eqnarray}
    \!\!\Gamm^{\overline{kl}}_{p,q,r}=\{T\in\C^{\times}_{p,q,r}:\; T\C^{\overline{k}}_{p,q,r}T^{-1}\subseteq\C^{\overline{kl}}_{p,q,r},\; T\C^{\overline{l}}_{p,q,r}T^{-1}\subseteq\C^{\overline{kl}}_{p,q,r}\}.\label{gkl_1}
\end{eqnarray}
Consider the case $(k,l)=(0,1)$ or $(2,3)$. By the statement (\ref{1_for_AB}) of Lemma \ref{lemma_for_AB}, we have
$\{T\in\C^{\times}_{p,q,r}:\;T\C^{\overline{k}}_{p,q,r}T^{-1}\subseteq\C^{\overline{kl}}_{p,q,r}\}=\{T\in\C^{\times}_{p,q,r}:\; \widetilde{T}T\in\Z^{\overline{k}\times}_{p,q,r}\}$,
therefore, 
\begin{eqnarray*}
    \Gamm^{\overline{kl}}_{p,q,r}\!=\!\{T\in\C^{\times}_{p,q,r}\!: \widetilde{T}T\in\Z^{\overline{k}\times}_{p,q,r},\widetilde{T}T\in\Z^{\overline{l}\times}_{p,q,r}\}\!=\!\{T\in\C^{\times}_{p,q,r}\!:\widetilde{T}T\in\Z^{\overline{kl}\times}_{p,q,r}\}.
\end{eqnarray*}
Thus, 
\begin{eqnarray*}
    \!\!\!\!\!\!\!&&\Gamm^{\overline{01}}_{p,q,r}=\{T\in\C^{\times}_{p,q,r}:\;\; \widetilde{T}T\in\Z^{\overline{01}\times}_{p,q,r}=\Z^{1\times}_{p,q,r}\}=\A^{\overline{01}}_{p,q,r},
    \\
    \!\!\!\!\!\!\!&&\Gamm^{\overline{23}}_{p,q,r}=\{T\in\C^{\times}_{p,q,r}:\;\;\widetilde{T}T\in\Z^{\overline{23}\times}_{p,q,r}=(\Z^{2}_{p,q,r}\cap\Z^3_{p,q,r})^{\times}\}=\A^{\overline{23}}_{p,q,r},
\end{eqnarray*}
where we use the formula (\ref{centralizers_qt_ds_1}) of Lemma \ref{centralizers_qt_ds}, and the proof is completed. It remains to consider the case $(k,l)=(1,2)$ or $(0,3)$. By the statement (\ref{2_for_AB}) of Lemma \ref{lemma_for_AB}, 
$\{T\in\C^{\times}_{p,q,r}:\; T\C^{\overline{k}}_{p,q,r}T^{-1}\subseteq\C^{\overline{kl}}_{p,q,r}\}=\{T\in\C^{\times}_{p,q,r}:\; \widehat{\widetilde{T}}T\in\Z^{\overline{k}}_{p,q,r}\}$,
so, from the equality (\ref{gkl_1}), we get $\Gamm^{\overline{kl}}_{p,q,r}=\{T\in\C^{\times}_{p,q,r}:\; \widehat{\widetilde{T}}T\in\Z^{\overline{kl}}_{p,q,r}\}$. Thus, using Lemma~\ref{centralizers_qt_ds}, we get
\begin{eqnarray*}
    &&\Gamm^{\overline{12}}_{p,q,r}=\{T\in\C^{\times}_{p,q,r}:\quad\widehat{\widetilde{T}}T\in\Z^{\overline{12}\times}_{p,q,r}=\Z^{1\times}_{p,q,r}\}=\B^{\overline{12}}_{p,q,r},
    \\
    &&\Gamm^{\overline{03}}_{p,q,r}=\{T\in\C^{\times}_{p,q,r}:\quad\widehat{\widetilde{T}}T\in\Z^{\overline{03}\times}_{p,q,r}=\Z^{3\times}_{p,q,r}\}=\B^{\overline{03}}_{p,q,r},
\end{eqnarray*}
and the proof is completed.
\end{proof}

\subsection{Groups $\check{\A}^{\overline{12}}_{p,q,r}$, $\check{\B}^{\overline{01}}_{p,q,r}$, $\check{\A}^{\overline{03}}_{p,q,r}$, $\check{\B}^{\overline{23}}_{p,q,r}$, $\check{\Gamm}^{\overline{01}}_{p,q,r}$, $\check{\Gamm}^{\overline{12}}_{p,q,r}$, $\check{\Gamm}^{\overline{23}}_{p,q,r}$, $\check{\Gamm}^{\overline{03}}_{p,q,r}$}\label{section_checkAB}

Consider the groups $\check{\A}^{\overline{12}}_{p,q,r}$, $\check{\A}^{\overline{03}}_{p,q,r}$, $\check{\B}^{\overline{01}}_{p,q,r}$, and  $\check{\B}^{\overline{23}}_{p,q,r}$ defined as follows:
\begin{eqnarray}
\!\!\!\!\!\!\!\!\!\!\!\!\!\!\!\check{\A}^{\overline{12}}_{p,q,r}&:=&\{T\in\C^{\times}_{p,q,r}:\quad\widetilde{T}T\in(\check{\Z}^{1}_{p,q,r}\cap\check{\Z}^{2}_{p,q,r})^{\times}\} \label{def_chA12}
\\
\!\!\!\!\!\!\!\!\!\!\!\!\!\!\!&=& \psi^{-1}((\check{\Z}^{1}_{p,q,r}\cap\check{\Z}^{2}_{p,q,r})^{\times}),
\\
\!\!\!\!\!\!\!\!\!\!\!\!\!\!\!\check{\A}^{\overline{03}}_{p,q,r}&:=&\{T\in\C^{\times}_{p,q,r}:\quad\widetilde{T}T\in(\Z^{3}_{p,q,r}\cap\C^{(0)}_{p,q,r})^{\times}\}\label{def_chA03}
\\
\!\!\!\!\!\!\!\!\!\!\!\!\!\!\!&=& \psi^{-1}((\Z^{3}_{p,q,r}\cap\C^{(0)}_{p,q,r})^{\times}),
\\
\!\!\!\!\!\!\!\!\!\!\!\!\!\!\!\check{\B}^{\overline{01}}_{p,q,r}&:=& \{T\in\C^{\times}_{p,q,r}\!: \quad\widehat{\widetilde{T}}T\!\in\!(\Z^1_{p,q,r}\cap\C^{(0)}_{p,q,r})^{\times}\!=\ker(\check{\ad})\!=\!\Lambd^{(0)\times}_r\!\} \label{def_chB01}
\\
\!\!\!\!\!\!\!\!\!\!\!\!\!\!\!&=&\chi^{-1}(\Lambd^{(0)\times}_r),
\\
\!\!\!\!\!\!\!\!\!\!\!\!\!\!\!\check{\B}^{\overline{23}}_{p,q,r}&:=&\{T\in\C^{\times}_{p,q,r}: \quad\widehat{\widetilde{T}}T\in(\check{\Z}^2_{p,q,r}\cap\check{\Z}^3_{p,q,r})^{\times}\} \label{def_chB23}
\\
\!\!\!\!\!\!\!\!\!\!\!\!\!\!\!&=& \chi^{-1}((\check{\Z}^2_{p,q,r}\cap\check{\Z}^3_{p,q,r})^{\times}),\label{def_chB23_}
\end{eqnarray}
where $\ker{(\check{\ad})}$ (\ref{ker_check_ad}) is the kernel of the twisted adjoint representation $\check{\ad}$ (\ref{twa1}),  the sets $\check{\Z}^{1}_{p,q,r}$, $\check{\Z}^{2}_{p,q,r}$ (\ref{ch_cc_2}), and $\check{\Z}^{3}_{p,q,r}$ (\ref{ch_cc_1}) are twisted centralizers of the subspaces $\C^{1}_{p,q,r}$, $\C^{2}_{p,q,r}$, and $\C^{3}_{p,q,r}$ respectively, and $\Z^1_{p,q,r}$ and $\Z^{3}_{p,q,r}$  are centralizers of the subspaces $\C^{1}_{p,q,r}$ and $\C^{3}_{p,q,r}$ respectively (see Lemma \ref{cases_we_use}). We have
\begin{eqnarray}
\!\!\!\!\!\!\!\!\!\!\!\!\!\!\!\!\!\!\!\!&&\check{\Z}^1_{p,q,r}\cap\check{\Z}^2_{p,q,r}=
\left\lbrace
    \begin{array}{lll}\label{chCC1_chCC2}
    \Lambd^{(0)}_r\oplus\Lambd^{n}_r,&&\mbox{$n$ is odd},
    \\
    \Lambd^{(0)}_r\oplus\Lambd^{n-1}_r,&&\mbox{$n$ is even};
    \end{array}
    \right.
    \\
\!\!\!\!\!\!\!\!\!\!\!\!\!\!\!\!\!\!\!\!&&\Z^3_{p,q,r}\cap\C^{(0)}_{p,q,r}=
\left\lbrace
            \begin{array}{lll}
            \Lambd^{(0)}_r\oplus \{\C^{1}_{p,q,0}\Lambd^{n-2}_r\} \oplus  \{\C^{2}_{p,q,0}\Lambd^{n-3}_r\}, \!\!\!\!\!&&\!\!\!\mbox{$n$ is odd},
            \\
            \Lambd^{(0)}_r\oplus \{\C^{1}_{p,q,0}\Lambd^{n-1}_r\}\oplus \{\C^{2}_{p,q,0}\Lambd^{n-2}_r\}, \!\!\!\!\!&&\!\!\!\mbox{$n$ is even};
            \end{array}
            \right.\label{chCC1_chCC3}
    \\
\!\!\!\!\!\!\!\!\!\!\!\!\!\!\!\!\!\!\!\!&&\check{\Z}^2_{p,q,r}\cap\check{\Z}^3_{p,q,r}=\left\lbrace
            \begin{array}{lll}
            \Lambd^{(0)}_r\oplus\Lambd^n_r\oplus \{\C^{1}_{p,q,0}\Lambd^{n-1}_r\},\!\!\!\!\!\!\!\!\!\!\!\!\!\!\!\!\!\!\!\!\!\!\!\!\!&&\!\!\!\!\!\!\!\!\!\!\mbox{$n$ is odd},
            \\
            \Lambd^{(0)}_r\oplus\Lambd^{n-1}_r\oplus\{\C^{1}_{p,q,0}(\Lambd^{n-2}_r\oplus\Lambd^{n-1}_r)\}
            \\
            \quad\oplus \{\C^{2}_{p,q,0}\Lambd^{n-2}_r\}, \!\!\!\!\!\!\!\!\!\!\!\!\!\!\!\!\!\!\!\!\!\!\!\!\!&&\!\!\!\!\!\!\!\!\!\!\mbox{$n$ is even}.
            \end{array}
            \right.\label{ch_cc2_ch_cc3}
\end{eqnarray}
 
The groups $\check{\A}^{\overline{12}}_{p,q,r}$, $\check{\A}^{\overline{03}}_{p,q,r}$, $\check{\B}^{\overline{01}}_{p,q,r}$, and $\check{\B}^{\overline{23}}_{p,q,r}$ generalize the groups $\A_{\pm}$ and $\B_{\pm}$ \cite{GenSpin}  respectively in the case of the non-degenerate Clifford algebras $\C_{p,q,0}$ to the case of the degenerate Clifford algebras $\C_{p,q,r}$ and coincide with them if $r=0$:
    \begin{eqnarray}
    &&\check{\A}^{\overline{12}}_{p,q,0}=\check{\A}^{\overline{03}}_{p,q,0}=\A_{\pm}=\{T\in\C^{\times}_{p,q,0}:\quad \widetilde{T}T\in\C^{0\times}\},\label{chAB_pq0_1}
    \\
    &&\check{\B}^{\overline{01}}_{p,q,0}=\check{\B}^{\overline{23}}_{p,q,0}=\B_{\pm}=\{T\in\C^{\times}_{p,q,0}:\quad \widehat{\widetilde{T}}T\in\C^{0\times}\}.\label{chAB_pq0_2}
    \end{eqnarray}

\begin{thm}\label{theorem_chAB}
In arbitrary $\C_{p,q,r}$, we have
\begin{eqnarray}
\!\!\!\!\!\!\!\!\!\!\!\!\!\!&\check{\A}^{\overline{12}}_{p,q,r}=\psi^{-1}((\check{\Z}^{1}_{p,q,r}\cap\check{\Z}^{2}_{p,q,r})^{\times})=\check{\Gamm}^{\overline{12}}_{p,q,r},\label{chAB_0}
\\
\!\!\!\!\!\!\!\!\!\!\!\!\!\!&\check{\A}^{\overline{03}}_{p,q,r}=\psi^{-1}((\Z^{3}_{p,q,r}\cap\C^{(0)}_{p,q,r})^{\times})=\check{\Gamm}^{\overline{03}}_{p,q,r},\label{chAB_1}
\\
\!\!\!\!\!\!\!\!\!\!\!\!\!\!\!\!\!&\check{\B}^{\overline{01}}_{p,q,r}\!=\chi^{-1}(\Lambd^{(0)\times}_r)=\check{\Gamm}^{\overline{01}}_{p,q,r}\subseteq \check{\B}^{\overline{23}}_{p,q,r}\!=\chi^{-1}((\check{\Z}^2_{p,q,r}\cap\check{\Z}^3_{p,q,r})^{\times})\!=\check{\Gamm}^{\overline{23}}_{p,q,r}.\label{chAB_2}
\end{eqnarray}
\end{thm}
\begin{proof}
The inclusion $\check{\B}^{\overline{01}}_{p,q,r}\subseteq\check{\B}^{\overline{23}}_{p,q,r}$ follows from the definitions of the groups (\ref{def_chB01}) and (\ref{def_chB23}) and explicit form of the sets $\check{\Z}^2_{p,q,r}$ (\ref{ch_cc_2}) and $\check{\Z}^3_{p,q,r}$ (\ref{ch_cc_1}). Let us prove the equalities in (\ref{chAB_0})--(\ref{chAB_2}). We have
\begin{eqnarray*}
    \check{\Gamm}^{\overline{kl}}_{p,q,r}=\{T\in\C^{\times}_{p,q,r}:\quad\widehat{T}\C^{\overline{k}}_{p,q,r}T^{-1}\subseteq\C^{\overline{kl}}_{p,q,r},\quad\widehat{T}\C^{\overline{l}}_{p,q,r}T^{-1}\subseteq\C^{\overline{kl}}_{p,q,r}\}.
\end{eqnarray*}
Consider the case $(k,l)=(0,1),(2,3)$. Since $\{T\in\C^{\times}_{p,q,r}:\; \widehat{T}\C^{\overline{k}}_{p,q,r}T^{-1}\subseteq\C^{\overline{kl}}_{p,q,r}\}=\{T\in\C^{\times}_{p,q,r}:\;\;\widehat{\widetilde{T}}T\in\check{\Z}^{\overline{k}\times}_{p,q,r}\}$ by Lemma \ref{lemma_for_AB}, we get
\begin{eqnarray*}
    \check{\Gamm}^{\overline{kl}}_{p,q,r}\!=\!\{T\in\C^{\times}_{p,q,r}\!: \widehat{\widetilde{T}}T\in\check{\Z}^{\overline{k}\times}_{p,q,r}, \widehat{\widetilde{T}}T\in\check{\Z}^{\overline{l}\times}_{p,q,r}\}\!=\!\{T\in\C^{\times}_{p,q,r}\!:\widehat{\widetilde{T}}T\in\check{\Z}^{\overline{kl}\times}_{p,q,r}\}.
\end{eqnarray*}
Thus,
\begin{eqnarray*}
    \!\!\!\!\!\!\!\!\!\!&&\check{\Gamm}^{\overline{01}}_{p,q,r}=\{T\in\C^{\times}_{p,q,r}:\; \widehat{\widetilde{T}}T\in\check{\Z}^{\overline{01}\times}_{p,q,r}=(\check{\Z}^1_{p,q,r}\cap\C^{(0)}_{p,q,r})^{\times}=\Lambd^{(0)\times}_r\}=\check{\B}^{\overline{01}}_{p,q,r},
    \\
    \!\!\!\!\!\!\!\!\!\!&&\check{\Gamm}^{\overline{23}}_{p,q,r}=\{T\in\C^{\times}_{p,q,r}:\; \widehat{\widetilde{T}}T\in\check{\Z}^{\overline{23}\times}_{p,q,r}=(\check{\Z}^2_{p,q,r}\cap\check{\Z}^3_{p,q,r})^{\times}\}=\check{\B}^{\overline{23}}_{p,q,r},
\end{eqnarray*}
where we use the statements (\ref{centralizers_qt_ds_2_0})--(\ref{centralizers_qt_ds_2}) of Lemma \ref{centralizers_qt_ds}. Now let us consider the case $(k,l)=(0,3),(1,2)$. Since $\{T\in\C^{\times}_{p,q,r}:\;\; \widehat{T}\C^{\overline{k}}_{p,q,r}T^{-1}\subseteq\C^{\overline{kl}}_{p,q,r}\}=\{T\in\C^{\times}_{p,q,r}:\;\; \widetilde{T}T\in\check{\Z}^{\overline{k}\times}_{p,q,r}\}$ by Lemma \ref{lemma_for_AB}, we obtain $\check{\Gamm}^{\overline{kl}}_{p,q,r}=\{T\in\C^{\times}_{p,q,r}:\;\; {\widetilde{T}}T\in\check{\Z}^{\overline{kl}\times}_{p,q,r}\}$. Thus,
\begin{eqnarray*}
    \!\!\!\!\!\!\!\!\!\!&&\check{\Gamm}^{\overline{03}}_{p,q,r}=\{T\in\C^{\times}_{p,q,r}:\quad {\widetilde{T}}T\in\check{\Z}^{\overline{03}\times}_{p,q,r}=(\Z^3_{p,q,r}\cap\C^{(0)}_{p,q,r})^{\times}\}=\check{\A}^{\overline{03}}_{p,q,r},
    \\
    \!\!\!\!\!\!\!\!\!\!&&\check{\Gamm}^{\overline{12}}_{p,q,r}=\{T\in\C^{\times}_{p,q,r}:\quad {\widetilde{T}}T\in\check{\Z}^{\overline{12}\times}_{p,q,r}=(\check{\Z}^1_{p,q,r}\cap\check{\Z}^2_{p,q,r})^{\times}\}=\check{\A}^{\overline{12}}_{p,q,r},
\end{eqnarray*}
where we use Lemma \ref{centralizers_qt_ds}, and the proof is completed.  
\end{proof}

\begin{rem}\label{rem43}
Note that
 \begin{eqnarray}\label{a12a23}
     \check{\A}^{\overline{12}}_{p,q,r}\subseteq\A^{\overline{23}}_{p,q,r},\quad \check{\B}^{\overline{01}}_{p,q,r}\subseteq\B^{\overline{12}}_{p,q,r}.
 \end{eqnarray} 
 Moreover, if  $n$ is even, then the groups coincide:
 \begin{eqnarray}
     \check{\A}^{\overline{12}}_{p,q,r} = \A^{\overline{23}}_{p,q,r},\quad \check{\B}^{\overline{01}}_{p,q,r}=\B^{\overline{12}}_{p,q,r},\qquad \mbox{$n$ is even}.
 \end{eqnarray}
\end{rem}
The relations between the groups $\check{\A}^{\overline{12}}_{p,q,r}$, $\check{\A}^{\overline{03}}_{p,q,r}$, $\A^{\overline{01}}_{p,q,r}$, $\A^{\overline{23}}_{p,q,r}$ and $\check{\B}^{\overline{01}}_{p,q,r}$, $\check{\B}^{\overline{23}}_{p,q,r}$, $\B^{\overline{12}}_{p,q,r}$, $\B^{\overline{03}}_{p,q,r}$ are studied in Section \ref{section_relations}.

\subsection{Groups $\tilde{\Q}^{\overline{12}}_{p,q,r}$, $\tilde{\Q}^{\overline{01}}_{p,q,r}$, $\tilde{\Q}^{\overline{03}}_{p,q,r}$, $\tilde{\Q}^{\overline{23}}_{p,q,r}$, $\tilde{\Gamm}^{\overline{01}}_{p,q,r}$, $\tilde{\Gamm}^{\overline{12}}_{p,q,r}$, $\tilde{\Gamm}^{\overline{23}}_{p,q,r}$, $\tilde{\Gamm}^{\overline{03}}_{p,q,r}$}\label{section_tildeAB}

Let us consider the following four groups:
\begin{eqnarray}
     \!\!\!\!\!\!\!\tilde{\Q}^{\overline{01}}_{p,q,r}&:=&\{T\in\C^{\times}_{p,q,r}: \quad\widetilde{T}T\in\Z^{4\times}_{p,q,r},\quad \widehat{\widetilde{T}}T\in\check{\Z}^{1\times}_{p,q,r}\}\label{def_tQ01}
     \\
     \!\!\!\!\!\!\!&=&\psi^{-1}(\Z^{4\times}_{p,q,r})\cap\chi^{-1}(\check{\Z}^{1\times}_{p,q,r}),
     \\
       \!\!\!\!\!\!\!\tilde{\Q}^{\overline{23}}_{p,q,r}&:=&\{T\in\C^{\times}_{p,q,r}: \quad\widetilde{T}T\in\Z^{2\times}_{p,q,r}, \quad\widehat{\widetilde{T}}T\in\check{\Z}^{3\times}_{p,q,r}\}\label{def_tQ23}
       \\
       \!\!\!\!\!\!\!&=&\psi^{-1}(\Z^{2\times}_{p,q,r})\cap\chi^{-1}(\check{\Z}^{3\times}_{p,q,r}),
     \\
    \!\!\!\!\!\!\!\tilde{\Q}^{\overline{12}}_{p,q,r}&:=&\{T\in\C^{\times}_{p,q,r}: \quad\widetilde{T}T\in\check{\Z}^{1\times}_{p,q,r}, \quad\widehat{\widetilde{T}}T\in\Z^{2\times}_{p,q,r}\}\label{def_tQ12}
    \\
    \!\!\!\!\!\!\!&=&\psi^{-1}(\check{\Z}^{1\times}_{p,q,r})\cap\chi^{-1}(\Z^{2\times}_{p,q,r}),
    \\
     \!\!\!\!\!\!\!\tilde{\Q}^{\overline{03}}_{p,q,r}&:=&\{T\in\C^{\times}_{p,q,r}: \quad\widetilde{T}T\in\check{\Z}^{3\times}_{p,q,r},\quad\widehat{\widetilde{T}}T\in\Z^{4\times}_{p,q,r}\}\label{def_tQ03}
     \\
     \!\!\!\!\!\!\!&=&\psi^{-1}(\check{\Z}^{3\times}_{p,q,r})\cap\chi^{-1}(\Z^{4\times}_{p,q,r}),\label{def_tQ03_}
\end{eqnarray}
where the sets $\Z^2_{p,q,r}$ and $\Z^4_{p,q,r}$ are centralizers of the subspaces $\C^{2}_{p,q,r}$ and $\C^{4}_{p,q,r}$ respectively,  and $\check{\Z}^1_{p,q,r}$  and $\check{\Z}^3_{p,q,r}$ (\ref{ch_cc_1}) are twisted centralizers of the subspaces $\C^{1}_{p,q,r}$ and $\C^{3}_{p,q,r}$ respectively (see Lemma~\ref{cases_we_use}).

\begin{thm}\label{theorem_tildeQ}
For the groups $($\ref{def_tQ01}$)$--$($\ref{def_tQ03_}$)$ in arbitrary $\C_{p,q,r}$, we have
\begin{eqnarray*}
    &\tilde{\Q}^{\overline{01}}_{p,q,r}=\tilde{\Gamm}^{\overline{01}}_{p,q,r},\qquad
    \tilde{\Q}^{\overline{23}}_{p,q,r}=\tilde{\Gamm}^{\overline{23}}_{p,q,r},
    \qquad\tilde{\Q}^{\overline{12}}_{p,q,r}=\tilde{\Gamm}^{\overline{12}}_{p,q,r}\subseteq \tilde{\Q}^{\overline{03}}_{p,q,r}=\tilde{\Gamm}^{\overline{03}}_{p,q,r}.
\end{eqnarray*}
\end{thm}
\begin{proof}
  We have $\tilde{\Q}^{\overline{12}}_{p,q,r}\subseteq\tilde{\Q}^{\overline{03}}_{p,q,r}$, since  $\check{\Z}^1_{p,q,r}\subseteq\check{\Z}^3_{p,q,r}$ and $\Z^2_{p,q,r}\subseteq\Z^4_{p,q,r}$ (see Lemma \ref{cases_we_use}).
     Also we have
    \begin{eqnarray}
    \!\!\!\tilde{\Gamma}^{\overline{kl}}_{p,q,r}=\{T\in\C^{\times}_{p,q,r}:\;\;T\C^{\overline{k}}_{p,q,r}T^{-1}\subseteq\C^{\overline{kl}}_{p,q,r},\;\widehat{T}\C^{\overline{l}}_{p,q,r}T^{-1}\subseteq\C^{\overline{kl}}_{p,q,r}\}
    \end{eqnarray}
    for $(k,l)=(0,1),(2,3),(2,1),(0,3)$.
    Using Lemmas  \ref{centralizers_qt_ds} and \ref{lemma_for_AB} and Remark \ref{cases_we_use}, we get
    \begin{eqnarray*}
\tilde{\Gamma}^{\overline{01}}_{p,q,r}&=&\{T\in\C^{\times}_{p,q,r}:\;\; \widetilde{T}T\in\Z^{\overline{0}\times}_{p,q,r}=\Z^{4\times}_{p,q,r},\;\;\widehat{\widetilde{T}}T\in\check{\Z}^{\overline{1}\times}_{p,q,r}=\check{\Z}^{1\times}_{p,q,r}=\Lambda^{\times}_r\}
         \\
         &=&\tilde{\Q}^{\overline{01}}_{p,q,r},
        \\
\tilde{\Gamma}^{\overline{23}}_{p,q,r}&=&\{T\in\C^{\times}_{p,q,r}:\;\; \widetilde{T}T\in\Z^{\overline{2}\times}_{p,q,r}=\Z^{2\times}_{p,q,r},\;\; \widehat{\widetilde{T}}T\in\check{\Z}^{\overline{3}\times}_{p,q,r}=\check{\Z}^{3\times}_{p,q,r}\}
        \\
        &=&\tilde{\Q}^{\overline{23}}_{p,q,r},
        \\
\tilde{\Gamma}^{\overline{12}}_{p,q,r}&=&\{T\in\C^{\times}_{p,q,r}:\;\; \widetilde{T}T\in\check{\Z}^{\overline{1}\times}_{p,q,r}=\check{\Z}^{1\times}_{p,q,r}=\Lambda^{\times}_r,\;\; \widehat{\widetilde{T}}T\in{\Z}^{\overline{2}\times}_{p,q,r}=\Z^{2\times}_{p,q,r}\}
        \\
        &=&\tilde{\Q}^{\overline{12}}_{p,q,r},
        \\
\tilde{\Gamma}^{\overline{03}}_{p,q,r}&=&\{T\in\C^{\times}_{p,q,r}:\;\; \widetilde{T}T\in\check{\Z}^{\overline{3}\times}_{p,q,r}=\check{\Z}^{3\times}_{p,q,r},\;\; \widehat{\widetilde{T}}T\in{\Z}^{\overline{0}\times}_{p,q,r}=\Z^{4\times}_{p,q,r}\}
        \\
        &=&\tilde{\Q}^{\overline{03}}_{p,q,r},
    \end{eqnarray*}
    and the proof is completed.
\end{proof}

\begin{rem}\label{remark_tildeQ}
The work \cite{GenSpin} considers the groups $\check{\Gamm}^{\overline{kl}}_{p,q,0}$ (\ref{ch_Gamma_kl}) preserving the direct sums of the subspaces $\C^{\overline{k}}_{p,q,0}$, $k=0,1,2,3$, under the twisted adjoint representation $\check{\ad}$ (\ref{twa1}) in the particular case of the non-degenerate Clifford algebras $\C_{p,q,0}$. However, that work does not consider the groups $\tilde{\Gamm}^{\overline{kl}}_{p,q,0}$ (\ref{ti_Gamma_kl}) preserving these subspaces under the twisted adjoint representation $\tilde{\ad}$ (\ref{twa22}). We provide the answer to this question in Theorem \ref{theorem_tildeQ} above. Note that by Lemma \ref{cases_we_use},
\begin{eqnarray}
    \!\!\!\!\!\!\!\!\!\!\!\!\!\!\!\!\!\!\!\!&&\tilde{\Q}^{\overline{01}}_{p,q,0}=\tilde{\Q}^{\overline{23}}_{p,q,0}=\{T\in\C^{\times}_{p,q,0}:\quad \widetilde{T}T\in\C^{0n\times}_{p,q,0},\quad \widehat{\widetilde{T}}T\in\C^{0\times}\},
    \\
    \!\!\!\!\!\!\!\!\!\!\!\!\!\!\!\!\!\!\!\!&&\tilde{\Q}^{\overline{12}}_{p,q,0}=\tilde{\Q}^{\overline{03}}_{p,q,0}=\{T\in\C^{\times}_{p,q,0}:\quad \widetilde{T}T\in\C^{0\times},\quad \widehat{\widetilde{T}}T\in\C^{0n\times}_{p,q,0}\},
\end{eqnarray}
Using $\widetilde{T}T\in\C^{\overline{01}}_{p,q,0}$ and $\widehat{\widetilde{T}}T\in\C^{\overline{03}}_{p,q,0}$ for any $T\in\C_{p,q,0}$ (\ref{norms_01_03}), we get that in the case $n=1,2,3\mod{4}$,  the groups $\tilde{\Q}^{\overline{01}}_{p,q,0}$, $\tilde{\Q}^{\overline{23}}_{p,q,0}$, $\tilde{\Q}^{\overline{12}}_{p,q,0}$, and $\tilde{\Q}^{\overline{03}}_{p,q,0}$ coincide with the groups $\Q$ and $\Q^{\pm}$ considered in \cite{OnInner} and \cite{GenSpin} respectively:
\begin{eqnarray}
\Q &=& \{T\in\C^{\times}_{p,q,0}:\quad \widetilde{T}T\in\Z^{\times}_{p,q,0},\quad\widehat{\widetilde{T}}T\in\Z^{\times}_{p,q,0}\},
\\
\Q^{\pm} &=& \{T\in\C^{\times}_{p,q,0}:\quad \widetilde{T}T\in\C^{0\times},\quad\widehat{\widetilde{T}}T\in\C^{0\times}\}
\end{eqnarray}
and preserving the subspaces $\C^{\overline{k}}_{p,q,0}$ (\ref{qtdef}), $k=0,1,2,3$, under $\ad$ (\ref{ar}) and $\check{\ad}$ (\ref{twa1}) respectively (see Theorem 6.3 \cite{OnInner} and Theorem 4 \cite{GenSpin})
\begin{eqnarray}
\!\!\!\!\!\!\!\!\!\!\!\!\!\!\!\!\!\!\!\!&&\tilde{\Q}^{\overline{01}}_{p,q,0}=\tilde{\Q}^{\overline{23}}_{p,q,0}=\Q,\qquad \tilde{\Q}^{\overline{12}}_{p,q,0}=\tilde{\Q}^{\overline{03}}_{p,q,0}=\Q^{\pm},\qquad  n=1\mod{4};
    \\
\!\!\!\!\!\!\!\!\!\!\!\!\!\!\!\!\!\!\!\!&&\tilde{\Q}^{\overline{01}}_{p,q,0}=\tilde{\Q}^{\overline{23}}_{p,q,0}=\tilde{\Q}^{\overline{12}}_{p,q,0}=\tilde{\Q}^{\overline{03}}_{p,q,0}=\Q^{\pm},\qquad n=2\mod{4};
\\
\!\!\!\!\!\!\!\!\!\!\!\!\!\!\!\!\!\!\!\!&&\tilde{\Q}^{\overline{01}}_{p,q,0}=\tilde{\Q}^{\overline{23}}_{p,q,0}=\Q^\pm,\qquad \tilde{\Q}^{\overline{12}}_{p,q,0}=\tilde{\Q}^{\overline{03}}_{p,q,0}=\Q,\qquad  n=3\mod{4}.
\end{eqnarray}
\end{rem}

\section{Relations between the groups}\label{section_relations}

This section studies the relations between the groups considered in Section~\ref{section_ds_qt}.
In Theorems \ref{lemma_AH}, \ref{theorem_rel_A}, and \ref{theorem_rel_B}, we prove that the four groups in the  families of type $\A$ ($\A^{\overline{01}}_{p,q,r}$ (\ref{def_A01}), $\A^{\overline{23}}_{p,q,r}$ (\ref{def_A23}), $\check{\A}^{\overline{12}}_{p,q,r}$ (\ref{def_chA12}), $\check{\A}^{\overline{03}}_{p,q,r}$ (\ref{def_chA03})) and type $\B$ ($\B^{\overline{12}}_{p,q,r}$ (\ref{def_B12}), $\B^{\overline{03}}_{p,q,r}$ (\ref{def_B03}), $\check{\B}^{\overline{01}}_{p,q,r}$ (\ref{def_chB01}), $\check{\B}^{\overline{23}}_{p,q,r}$ (\ref{def_chB23}))
respectively  
are related to each other through multiplication.

In this section, we use the notion of the Jacobson radical $\rad\C_{p,q,r}$ of the algebra $\C_{p,q,r}$. 
Let $A,B,C$ be ordered multi-indices with non-zero length and $e_A=e_{a_1}\cdots e_{a_k}$ with $\{a_1,\ldots,a_k\}\subseteq\{1,\ldots,p\}$, $e_B=e_{b_1}\cdots e_{b_l}$ with $\{b_1,\ldots,b_l\}\subseteq\{p+1,\ldots,p+q\}$, $e_C=e_{c_1}\cdots e_{c_m}$ with $\{c_1,\ldots,c_m\}\subseteq\{p+q+1,\ldots,n\}$.
An arbitrary element $y\in\rad\C_{p,q,r}$ has the form
\begin{eqnarray}\label{rad}
y=\sum_C v_C e_C +\sum_{A,C} v_{AC} e_A e_C+\sum_{B,C} v_{BC} e_B e_C +\sum_{A,B,C} v_{ABC}e_{A}e_{B}e_{C},
\end{eqnarray}
where $v_C,v_{AC},v_{BC},v_{ABC}\in\F$.
In this section, we need auxiliary Lemma \ref{lemma_inv_rad}, which is proved, for example, in \cite{OnSomeLie}, and Remarks \ref{remark_inv_l0} and \ref{remark_commute_L0}.

\begin{lem}\label{lemma_inv_rad}
An element $T\in\C^{0}\oplus\rad\C_{p,q,r}$ is invertible if and only if its projection onto the  grade $0$ subspace is non-zero:
    \begin{eqnarray}
        T\in\C^{0}\oplus\rad\C_{p,q,r},\quad \langle T\rangle_0\neq 0\qquad\Leftrightarrow\qquad T\in(\C^{0}\oplus\rad\C_{p,q,r})^{\times}.
    \end{eqnarray}
\end{lem}

To prove Theorem \ref{lemma_AH}, we also need Remarks \ref{remark_inv_l0} and \ref{remark_commute_L0}.

\begin{rem}\label{remark_inv_l0}
    Any element $X\in\Lambda^{\overline{k}}_r$, $k=0,1,2,3$, with $\langle X
    \rangle_{0}=0$ is nilpotent, i.e., there exists such $N\in\mathbb{N}$ that $X^N=0$. Therefore, the inverse of any
    \begin{eqnarray}
    \alpha e+ X,\qquad \alpha\in\F^{\times},\quad X\in\Lambda^{\overline{0}}_{r},\quad \langle X\rangle_0=0,
    \end{eqnarray}
    has the form
    \begin{eqnarray*}
        (\alpha e+ X)^{-1} = \frac{1}{\alpha}\left(e-\frac{1}{\alpha}X+\frac{1}{\alpha^2}X^2+\cdots+(-1)^{N-1}\frac{1}{\alpha^{N-1}}X^{N-1}\right)\in\Lambda^{\overline{0}\times}_r.
    \end{eqnarray*}
\end{rem}

\begin{rem}\label{remark_commute_L0}
Note that 
\begin{eqnarray}
\{X\in\C_{p,q,r}:\quad XV=VX,\quad \forall V\in\Lambda^{(0)}_r\}=\C_{p,q,r},\label{f_00}
\end{eqnarray}
because for any even basis element $e_{r_1\ldots r_l}\in\Lambda^{(0)}_r$, where $r_{1},\ldots, r_l\in\{p+q+1,\ldots,n\}$, and for any basis element $e_{a_1\ldots a_k}\in\C_{p,q,r}$, where $a_1,\ldots, a_k\in\{1,\ldots,n\}$, we have 
\begin{eqnarray}
e_{a_1\ldots a_k}e_{r_1\ldots r_l} = 
\left\lbrace
                \begin{array}{lll}
                e_{r_1\ldots r_l}e_{a_1\ldots a_k}, && \{a_1\ldots a_k\}\cap \{r_1\ldots r_l\}=\varnothing,
                \\
                0,&& \mbox{otherwise},
                \end{array}
                \right.
\end{eqnarray}
and we get (\ref{f_00}) by linearity.
\end{rem}

Consider the following simple auxiliary group:
\begin{eqnarray}
    \!\!\!\check{\A}_{p,q,r}:=\{T\in\C^{\times}_{p,q,r}:\quad \widetilde{T}T\in(\Z^1_{p,q,r}\cap\C^{(0)}_{p,q,r})^{\times}=\ker(\check{\ad})=\Lambd^{(0)\times}_r\},\label{def_chA}
\end{eqnarray}
where $\Z^1_{p,q,r}=\Z_{p,q,r}$ is the centralizer (\ref{cc_2}) of the grade-$1$ subspace and $\ker(\check{\ad})$ is the kernel (\ref{ker_check_ad}) of the twisted adjoint representation $\check{\ad}$ (\ref{twa1}). Note that $\check{\A}_{p,q,r}$ is a subgroup of the groups (\ref{def_A01}), (\ref{def_A23}), (\ref{def_chA12}), and (\ref{def_chA03}):
\begin{eqnarray}
\check{\A}_{p,q,r}\subseteq \A^{\overline{01}}_{p,q,r},\quad \check{\A}_{p,q,r}\subseteq \A^{\overline{23}}_{p,q,r},\quad 
\check{\A}_{p,q,r}\subseteq \check{\A}^{\overline{12}}_{p,q,r},\quad 
\check{\A}_{p,q,r}\subseteq \check{\A}^{\overline{03}}_{p,q,r}.
\end{eqnarray}
We have by (\ref{norms_01_03}),
\begin{eqnarray}
\!\!\!\!\!\!\!\!\!\!\!\!\!\!\!&\check{\A}_{p,q,r}=\A^{\overline{01}}_{p,q,r},\quad n=0,2,3\mod{4}; \;\quad \check{\A}_{p,q,r}=\A^{\overline{23}}_{p,q,r},\quad n=0\mod{4};\label{chA_c_1}
\\
\!\!\!\!\!\!\!\!\!\!\!\!\!\!\!&\check{\A}_{p,q,r}=\check{\A}^{\overline{12}}_{p,q,r},\quad n=0,3\mod{4};\;\quad \check{\A}_{p,q,r}=\check{\A}^{\overline{03}}_{p,q,r},\quad n=2,3\mod{4}.\label{chA_c_2}
\end{eqnarray}

Similarly to the group $\check{\A}_{p,q,r}$ (\ref{def_chA}), let us use the following notation for a simple auxiliary group of type $\B$:
\begin{eqnarray}
    \check{\B}_{p,q,r}:=\{T\in\C^{\times}_{p,q,r}:\quad \widehat{\widetilde{T}}T\in\Lambd^{(0)\times}_r\}.\label{def_chB}
\end{eqnarray}
The group $\check{\B}_{p,q,r}$ coincides with the group $\check{\B}^{\overline{01}}_{p,q,r}$ (\ref{def_chB01}):
\begin{eqnarray}
\check{\B}_{p,q,r}=\check{\B}^{\overline{01}}_{p,q,r}. 
\end{eqnarray}
In Theorem \ref{lemma_AH} and below, we denote by $\langle W\rangle_{\Lambda^{(0)}_r}$ the projection of a multivector $W\in\C_{p,q,r}$ onto the subspace $\Lambda^{(0)}_r$.

\begin{thm}\label{lemma_AH}
    Suppose $\H$ is a vector space and a monoid under multiplication, $\Lambda^{(0)}_r\subseteq\H\subseteq\C^0\oplus\rad\C_{p,q,r}$, and $(W - \langle W\rangle_{\Lambda^{(0)}_r})^2=0$  for any $W\in\H$. Then 
    \begin{eqnarray}
    \{T\in\C^{\times}_{p,q,r}:\quad {\widetilde{T}}T\in \H^{\times}\} &=& \check{\A}_{p,q,r} \H^{\times},\label{lem_chA_}
    \\
    \{T\in\C^{\times}_{p,q,r}:\quad \widehat{\widetilde{T}}T\in \H^{\times}\} &=& \check{\B}_{p,q,r} \H^{\times}.\label{lem_chB_}
    \end{eqnarray}
\end{thm}
\begin{proof}
Let us prove that the right-hand side is a subset of the left-hand side in (\ref{lem_chA_}). Suppose $T\in\check{\A}_{p,q,r}$ and $Y\in\H^{\times}$; then $\widetilde{(TY)}TY=\widetilde{Y}\widetilde{T}TY\in\H^{\times}$, because $\widetilde{T}T\in\Lambda^{(0)\times}_r\subseteq\H$ and $\H$ is a monoid. Thus, $TY\in\{T\in\C^{\times}_{p,q,r}:\quad {\widetilde{T}}T\in \H^{\times}\}$. 
Similarly, the right-hand side is a subset of the left-hand side in (\ref{lem_chB_}), because for $T\in\check{\B}_{p,q,r}$ and $Y\in\H^{\times}$, we have $\widehat{\widetilde{(TY)}}TY=\widehat{\widetilde{Y}}\widehat{\widetilde{T}}TY\in\H^{\times}$, since $\widehat{\widetilde{T}}T\in\Lambda^{(0)\times}_r\subseteq\H$.

Let us prove that the left-hand side is a subset of the right-hand side in (\ref{lem_chA_}).
Suppose $T\in\C^{\times}_{p,q,r}$ and $\widetilde{T}T=\alpha e+X+ W$, where $\alpha\in\F$, $X\in\Lambda^{(0)}_r$, $\langle X\rangle_0=0$, $W\in\H$, and $\langle W\rangle_{\Lambda^{(0)}_r}=0$.
Note that $\widetilde{T}T\in(\C^0\oplus\rad\C_{p,q,r})^{\times}$; therefore, $\alpha\neq0$ and $\alpha e+X$ is invertible by Lemma  \ref{lemma_inv_rad}. 
We need to show that there exists such $Y\in\H^{\times}$ that $TY^{-1}\in\check{\A}_{p,q,r}$, and the statement will be proved.
Consider 
\begin{eqnarray}
    Y = e + \frac{1}{2}(\alpha e + X)^{-1}W.\label{y1}
\end{eqnarray}
Note that $Y\in\H$ because $(\alpha e+X)^{-1}\in\Lambda^{(0)\times}_r\subset\H$, $W\in\H$, and $\H$ is a multiplicative monoid and linear space.
 $Y$ is invertible and 
\begin{eqnarray}
    Y^{-1} = e - \frac{1}{2}(\alpha e + X)^{-1}W,\label{y-1}
\end{eqnarray}
because
\begin{eqnarray}
    Y Y^{-1}&=& e -\frac{1}{4}(\alpha e + X)^{-1}W(\alpha e + X)^{-1}W
    \\
    &=&e -\frac{1}{4}((\alpha e + X)^{-1})^2W^2=e,
\end{eqnarray}
where we use $W(\alpha e + X)^{-1}=(\alpha e + X)^{-1}W$ by Remark \ref{remark_commute_L0} and apply the assumption $W^2=0$.
Also note that $\widetilde{Y}=Y$. This fact holds since $\widetilde{T}T\in\C^{\overline{01}}_{p,q,r}$ by (\ref{norms_01_03}); so, $X\in\Lambda^{\overline{0}}_r$, $W\in\C^{\overline{01}}_{p,q,r}$, therefore, $(\alpha e +X)^{-1}\in\Lambda^{\overline{0}}_r$ (Remark \ref{remark_inv_l0}), and $(\alpha e +X)^{-1}W\in\C^{\overline{01}}_{p,q,r}$.
Finally, we get
\begin{eqnarray}
   \!\!\!\!\!\!\!\!\!\!\!\!\!\!\!\!\!\!\!\! &&\widetilde{(TY^{-1})}(TY^{-1}) = Y^{-1}\widetilde{T}TY^{-1} 
    \\
  \!\!\!\!\!\!\!\!\!\!\!\!\!\!\! \!\!\!\!\! &&\quad = 
    \left(e - \frac{1}{2}(\alpha e + X)^{-1}W\right)\left(\alpha e+X+ W\right)\left(e - \frac{1}{2}(\alpha e + X)^{-1}W\right)\label{y2}
    \\
    \!\!\!\!\!\!\!\!\!\!\!\!\!\!\!\!\!\!\!\!&&\quad =\alpha e +X+W -\alpha(\alpha e+X)^{-1}W - (\alpha e+X)^{-1}WX\label{t_rel_f0}
    \\
   \!\!\!\!\!\!\!\!\!\!\!\!\!\!\!\!\!\!\!\! &&\quad  = (\alpha e+X)^{-1}\left((\alpha e+X)^2+(\alpha e+X)W-\alpha W - WX\right)
    \\
    \!\!\!\!\!\!\!\!\!\!\!\!\!\!\!\!\!\!\!\!&&\quad =\alpha e + X\in\Lambda^{(0)\times}_r,\label{t_rel_f1}
\end{eqnarray}
where in (\ref{t_rel_f0}), we use by Remark \ref{remark_commute_L0},
\begin{eqnarray*}
    W(\alpha e + X)^{-1}\!=\!(\alpha e + X)^{-1}W,\; XW\!=\!WX,\; X(\alpha e+X)^{-1}\!=\!(\alpha e+X)^{-1}X
\end{eqnarray*}
and $W^2=0$; and in (\ref{t_rel_f1}), we apply $XW=WX$ again. So, $TY^{-1}\in\check{\A}_{p,q,r}$ by definition (\ref{def_chA}), and the proof is completed.

Similarly, we prove that the left-hand side is a subset of the right-hand side in (\ref{lem_chB_}). Suppose $T\in\C^{\times}_{p,q,r}$ and $\widehat{\widetilde{T}}T=\alpha e+X+ W$, where $\alpha\in\F$, $X\in\Lambda^{(0)}_r$, $\langle X\rangle_0=0$, $W\in\H$, and $\langle W\rangle_{\Lambda^{(0)}_r}=0$.
Since $\widehat{\widetilde{T}}T\in(\C^0\oplus\rad\C_{p,q,r})^{\times}$; we have $\alpha\neq0$ and $\alpha e+X$ is invertible by Lemma  \ref{lemma_inv_rad}. 
We need to show that there exists such $Y\in\H^{\times}$ that $TY^{-1}\in\check{\B}_{p,q,r}$. Consider $Y$ from the formula (\ref{y1}) with the inverse (\ref{y-1}). Note that $\widehat{\widetilde{Y}}=Y$, because $\widehat{\widetilde{T}}T\in\C^{\overline{03}}_{p,q,r}$ by (\ref{norms_01_03}), so, $X\in\Lambda^{\overline{0}}_r$, $W\in\C^{\overline{03}}_{p,q,r}$, therefore, $(\alpha e +X)^{-1}\in\Lambda^{\overline{0}}_r$ (Remark \ref{remark_inv_l0}), and $(\alpha e +X)^{-1}W\in\C^{\overline{03}}_{p,q,r}$.
Finally,
\begin{eqnarray}
\widehat{\widetilde{(TY^{-1})}}(TY^{-1}) = Y^{-1}\widehat{\widetilde{T}}TY^{-1} =\alpha e + X\in\Lambda^{(0)\times}_r,
\end{eqnarray}
where the last equality is obtained by the formulas (\ref{y2})--(\ref{t_rel_f1}). So, $TY^{-1}\in\check{\B}^{\overline{01}}_{p,q,r}$, and this completes the proof.
\end{proof}

In Lemma \ref{lemma_mult_c3}, let us consider several sets that satisfy the assumptions of Theorem \ref{lemma_AH}.  
We need Remark \ref{rem_uv} and Lemma \ref{TT_12} to prove this lemma.

\begin{rem}\label{rem_uv}
For any $k,l=0,1,\ldots,n$,
\begin{eqnarray}
UV\in\Lambd^{k+l}_r,\qquad \forall U\in\Lambd^{k}_r,\qquad \forall V\in\Lambd^{l}_r.\label{remark_klkl}
\end{eqnarray}
\end{rem}

\begin{lem}\label{TT_12}
We have
\begin{eqnarray*}
T^{2}=0\quad \mbox{for any}\quad T\in\Lambd^{k}_r,\quad k=\left\lbrace
    \begin{array}{lll}
    \;[\frac{n}{2}]+1,\ldots, n&& \mbox{if $n=0,1\mod{4}$},
    \\
    \;[\frac{n}{2}],\ldots, n&&\mbox{if $n=2,3\mod{4}$}.
     \end{array}
    \right.
\end{eqnarray*}
\end{lem}
\begin{proof}
    Suppose $T\in\Lambda^{k}_r$; then $T^2\in\Lambda^{2k}_r$  (\ref{remark_klkl}). If $k\geq [\frac{n}{2}]+1$, then $2k>n$, and we get $T^2=0$. On the one hand, for $k=[\frac{n}{2}]$, we have $T^2\in\Lambda^n_r$ if $n=2\mod{4}$ and $T^2\in\Lambda^{n-1}_r$ if $n=3\mod{4}$ (\ref{remark_klkl}). On the other hand, $T^2=\pm\widetilde{T}T\in\C^{\overline{01}}_{p,q,r}$ by (\ref{norms_01_03}); therefore, $T^2=0$. 
\end{proof}

Let us use the following notation for any fixed set $\H\subseteq\C_{p,q,r}$ and its subset $\D\subseteq\H$. The set $\H$ can be represented as a direct sum of $\D$ and its orthogonal complement $\D^{\perp}_{\H}$:
\begin{eqnarray}
\H = \D\oplus \D^{\perp}_{\H},
\end{eqnarray}
where 
\begin{eqnarray}\label{def_orthcomp}
    \D^{\perp}_{\H}:= \{X\in \H:\quad \langle X\rangle_{\D}=0\}\subseteq \H.
\end{eqnarray}
For example, $\big({\Lambda^{(0)}_r}\big)_{\Z^{m}_{p,q,r}}^{\perp} := \{X\in \Z^{m}_{p,q,r}:\quad \langle X\rangle_{\Lambda^{(0)}_r}=0\}$.

\begin{lem}\label{lemma_mult_c3}
In the case $r\neq0$, we have
\begin{eqnarray}
\!\!\!\!\!\!\!\!\!\!\!\!\!\!\!\!\!\!\!\!\!\!\!\!\!&&\left(\big({\Lambda^{(0)}_r}\big)^{\perp}_{\Z_{p,q,r}}\right)^2=\{0\}\quad\mbox{if $n$ is odd};\label{f-1}
\\
\!\!\!\!\!\!\!\!\!\!\!\!\!\!\!\!\!\!\!\!\!\!\!\!\!&&\left(\big(\Lambda^{(0)}_r\big)^{\perp}_{\Z^{3}_{p,q,r}}\right)^2=\{0\}\quad\mbox{if} \;\; n=4,\;n\geq6;\label{f_0}
\\
\!\!\!\!\!\!\!\!\!\!\!\!\!\!\!\!\!\!\!\!\!\!\!\!\!&&\left(\big(\Lambd^{(0)}_{r}\big)^{\perp}_{\check{\Z}^2_{p,q,r}\cap\check{\Z}^3_{p,q,r}}\right)^2=\left(\big(\Lambd^{(0)}_{r}\big)^{\perp}_{\check{\Z}^1_{p,q,r}\cap\check{\Z}^2_{p,q,r}}\right)^2=\{0\}\quad\mbox{if}\;\;n\geq3;\label{f_3}
\\
\!\!\!\!\!\!\!\!\!\!\!\!\!\!\!\!\!\!\!\!\!\!\!\!\!&& \left(\big(\Lambd^{(0)}_{r}\big)^{\perp}_{{\Z}^2_{p,q,r}\cap{\Z}^3_{p,q,r}}\right)^2=\left(\big(\Lambd^{(0)}_{r}\big)^{\perp}_{{\Z}^3_{p,q,r}\cap\C^{(0)}_{p,q,r}}\right)^2=\{0\}\quad\mbox{if}\;\;n\geq4.\label{f_5}
\end{eqnarray}
\end{lem}
\begin{proof} We have (\ref{f-1}) because in the case of odd $n$, $\big(\Lambd^{(0)}_{r}\big)^{\perp}_{\Z_{p,q,r}}=\C^n_{p,q,r}$. 
    Let us prove (\ref{f_0}). Suppose $n\geq 4$ is even. Then by (\ref{remark_klkl}), 
    \begin{eqnarray}
        \!\!\!\!\!\!\!\!\!\!\!\!\!\!\!\!\!\!\!\!\!\!\!\!\left(\big(\Lambd^{(0)}_{r}\big)^{\perp}_{\Z^3_{p,q,r}}\right)^2&=&(\Lambda^{n-1}_r\oplus\{\C^1_{p,q,0}\Lambda_r^{\geq n-2}\}\oplus\{\C_{p,q,0}^2\Lambda_r^{n-2}\})^2
        \\
        &=&(\{\C^1_{p,q,0}\Lambda_r^{ n-2}\})^2\oplus(\{\C_{p,q,0}^2\Lambda_r^{n-2}\})^2=\{0\}.
    \end{eqnarray}
    Suppose $n\geq7$ is odd. Again by (\ref{remark_klkl}), 
    \begin{eqnarray*}
        \left(\big(\Lambd^{(0)}_{r}\big)^{\perp}_{\Z^3_{p,q,r}}\right)^2 \!\!\!\!\!&=& \!\!\!\!\!(\Lambd^{n-2}_r\oplus \{\C^{1}_{p,q,0}(\Lambd^{n-3}_r\oplus\Lambd^{n-2}_r)\}\oplus\{\C^{2}_{p,q,0}\Lambd^{n-3}_r\}\!\oplus\!\C^{n}_{p,q,r})^2
        \\
        &=& (\{\C^{1}_{p,q,0}(\Lambd^{n-3}_r\oplus\Lambd^{n-2}_r)\})^2\oplus(\{\C^{2}_{p,q,0}\Lambd^{n-3}_r\})^2= \{0\}.
    \end{eqnarray*}

    Now let us prove (\ref{f_3}) and (\ref{f_5}). Since $(\check{\Z}^2_{p,q,r}\cap\check{\Z}^3_{p,q,r})\subseteq\Z^3_{p,q,r}$ (\ref{ch_cc2_ch_cc3}), $(\check{\Z}^1_{p,q,r}\cap\check{\Z}^2_{p,q,r})\subseteq\Z^3_{p,q,r}$ (\ref{chCC1_chCC2}), $({\Z}^2_{p,q,r}\cap{\Z}^3_{p,q,r})\subseteq\Z^3_{p,q,r}$ (\ref{cc2_cap_cc3}), and $({\Z}^3_{p,q,r}\cap\C^{(0)}_{p,q,r})\subseteq\Z^3_{p,q,r}$ (\ref{chCC1_chCC3}), we automatically get the equalities (\ref{f_3}) and (\ref{f_5}) in the case $n=4$, $n\geq6$ and $r\neq0$ by (\ref{f_0}). Consider the case $n=3$ or $5$ and $r\neq 0$. We get by (\ref{remark_klkl}),
\begin{eqnarray}
\left(\big(\Lambd^{(0)}_{r}\big)^{\perp}_{\check{\Z}^2_{p,q,r}\cap\check{\Z}^3_{p,q,r}}\right)^2 =(\Lambda_r^{n}\oplus \{\C^1_{p,q,0}\Lambda^{n-1}_r\})^2 =\{0\}.\label{f_6}
\end{eqnarray} 
Using $\check{\Z}^1_{p,q,r}\cap\check{\Z}^2_{p,q,r}\subseteq\check{\Z}^2_{p,q,r}\cap\check{\Z}^3_{p,q,r}$ and (\ref{f_6}), we obtain (\ref{f_3}).
In the case $n=5$, we get by (\ref{remark_klkl}),
\begin{eqnarray}
\!\!\!\!\!\!\!\!\!\!\!\!\!\!\!\!\!\!&&\left(\big(\Lambd^{(0)}_{r}\big)^{\perp}_{{\Z}^2_{p,q,r}\cap{\Z}^3_{p,q,r}}\right)^2=(\Lambda^{3}_r\oplus\C^5_{p,q,r})^2
=\{0\},
\\
\!\!\!\!\!\!\!\!\!\!\!\!\!\!\!\!\!\!&&\left(\big(\Lambd^{(0)}_{r}\big)^{\perp}_{{\Z}^3_{p,q,r}\cap\C^{(0)}_{p,q,r}}\right)^2=(\{\C^{1}_{p,q,0}\Lambda^{3}_r\}\oplus\{\C^{2}_{p,q,0}\Lambda^{2}_r\})^2=\{0\},
\end{eqnarray}
and (\ref{f_5}) is proved.
\end{proof}

\begin{thm}\label{theorem_rel_A}
The groups $\A^{\overline{23}}_{p,q,r}$, $\A^{\overline{01}}_{p,q,r}$, $\check{\A}^{\overline{12}}_{p,q,r}$, and   $\check{\A}^{\overline{03}}_{p,q,r}$ are related in the following way in the degenerate geometric algebras $\C_{p,q,r}$, $r\neq0$:
\begin{eqnarray}
\!\!\!\!\!\!\!\!\!\!&\check{\A}^{\overline{03}}_{p,q,r}=\check{\A}_{p,q,r}({\Z}^3_{p,q,r}\cap\C^{(0)}_{p,q,r})^{\times},\label{l_a23_0}
\\
\!\!\!\!\!\!\!\!\!\!&\check{\A}^{\overline{12}}_{p,q,r}=\check{\A}_{p,q,r}(\check{\Z}^1_{p,q,r}\cap\check{\Z}^2_{p,q,r})^{\times},\quad \A^{\overline{23}}_{p,q,r}= \A^{\overline{01}}_{p,q,r}(\Z^2_{p,q,r}\cap\Z^3_{p,q,r})^{\times},\label{l_a23_1}
\\
\!\!\!\!\!\!\!\!\!\!& \A^{\overline{01}}_{p,q,r} = \check{\A}_{p,q,r}\Z^{\times}_{p,q,r}.\label{chA_l}
\end{eqnarray}
\end{thm}
\begin{proof} 
    Let us prove (\ref{l_a23_0}).
    If $n\leq3$, then $\check{\A}^{\overline{03}}_{p,q,r}=\check{\A}_{p,q,r}$ by (\ref{chCC1_chCC3}) and (\ref{chA_c_2}). 
    We have $\check{\A}_{p,q,r}({\Z}^3_{p,q,r}\cap\C^{(0)}_{p,q,r})^{\times}\subseteq\check{\A}^{\overline{03}}_{p,q,r}$ because $\widetilde{(TY)}TY = \widetilde{Y}\widetilde{T}TY\in({\Z}^3_{p,q,r}\cap\C^{(0)}_{p,q,r})^{\times}$ for any $T\in\check{\A}_{p,q,r}$ and $Y\in({\Z}^3_{p,q,r}\cap\C^{(0)}_{p,q,r})^{\times}$ using $\widetilde{T}T\in\Lambda^{(0)\times}_r\subseteq ({\Z}^3_{p,q,r}\cap\C^{(0)}_{p,q,r})^{\times}$.
    In the case $n\geq4$, we have $(W-\langle W\rangle_{\Lambda^{(0)}_r})^2=0$ for any $W\in{\Z}^3_{p,q,r}\cap\C^{(0)}_{p,q,r}$ by Lemma \ref{lemma_mult_c3} and $\Lambda^{(0)}_r\subseteq({\Z}^3_{p,q,r}\cap\C^{(0)}_{p,q,r})\subseteq\C^0\oplus\rad\C_{p,q,r}$. Thus, by Theorem \ref{lemma_AH},
    \begin{eqnarray*}
        \check{\A}^{\overline{03}}_{p,q,r} =\{T\in\C^{\times}_{p,q,r}:\; \widetilde{T}T\in({\Z}^3_{p,q,r}\cap\C^{(0)}_{p,q,r})^{\times}\}=\check{\A}_{p,q,r}({\Z}^3_{p,q,r}\cap\C^{(0)}_{p,q,r})^{\times}.
    \end{eqnarray*}
    
    Now we prove (\ref{l_a23_1}). Note that 
    $\Lambda^{(0)}_r\subseteq(\check{\Z}^1_{p,q,r}\cap\check{\Z}^2_{p,q,r})\subseteq({\Z}^2_{p,q,r}\cap\Z^3_{p,q,r})\subseteq\C^0\oplus\rad\C_{p,q,r}$.
    If $n=1,2$, then $\big(\Lambd^{(0)}_{r}\big)^{\perp}_{\check{\Z}^1_{p,q,r}\cap\check{\Z}^2_{p,q,r}}=\big(\Lambd^{(0)}_{r}\big)^{\perp}_{{\Z}^2_{p,q,r}\cap{\Z}^3_{p,q,r}}=\Lambda^1_r$ and $W^2=0$ for any $W\in\Lambda^1_r$ by Lemma \ref{TT_12}.
    In the case $n\geq 3$, we have $W^2=0$ for any $W\in\big(\Lambd^{(0)}_{r}\big)^{\perp}_{\check{\Z}^1_{p,q,r}\cap\check{\Z}^2_{p,q,r}}$ or $W\in \big(\Lambd^{(0)}_{r}\big)^{\perp}_{{\Z}^2_{p,q,r}\cap{\Z}^3_{p,q,r}}$ by Lemma \ref{lemma_mult_c3}. Thus, by Theorem \ref{lemma_AH},
    \begin{eqnarray*}
         \!\!\!\!\! \!\!\!\!\!\!\!\!\!\!&&\check{\A}^{\overline{12}}_{p,q,r} =\{T\in\C^{\times}_{p,q,r}:\; \widetilde{T}T\in(\check{\Z}^1_{p,q,r}\cap\check{\Z}^2_{p,q,r})^{\times}\}=\check{\A}_{p,q,r}(\check{\Z}^1_{p,q,r}\cap\check{\Z}^2_{p,q,r})^{\times},
        \\
         \!\!\!\!\! \!\!\!\!\!\!\!\!\!\!&&{\A}^{\overline{23}}_{p,q,r} =\{T\in\C^{\times}_{p,q,r}:\; \widetilde{T}T\in({\Z}^2_{p,q,r}\cap{\Z}^3_{p,q,r})^{\times}\}=\check{\A}_{p,q,r}({\Z}^2_{p,q,r}\cap{\Z}^3_{p,q,r})^{\times}.
    \end{eqnarray*}

    Finally, we have (\ref{chA_l}) because in the case of odd $n$, we have $W^2=0$ for any $W\in\big(\Lambd^{(0)}_{r}\big)^{\perp}_{\Z_{p,q,r}}$ by Lemma \ref{lemma_mult_c3} and in the case of even $n$, $\big(\Lambd^{(0)}_{r}\big)^{\perp}_{\Z_{p,q,r}}=\varnothing$, so we can apply Theorem \ref{lemma_AH}:
    \begin{eqnarray}
        \check{\A}^{\overline{01}}_{p,q,r} =\{T\in\C^{\times}_{p,q,r}:\; \widetilde{T}T\in\Z^{\times}_{p,q,r}\}=\check{\A}_{p,q,r}\Z^{\times}_{p,q,r},
    \end{eqnarray}
    and the proof is completed.
\end{proof}

\begin{thm}\label{theorem_rel_B}
    We have the following relations between the groups in the degenerate geometric algebras $\C_{p,q,r}$, $r\neq0$:
    \begin{eqnarray}
        &\B^{\overline{12}}_{p,q,r} = \check{\B}_{p,q,r}\Z^{\times}_{p,q,r},\label{chB_l}
        \\
    &
    \check{\B}^{\overline{23}}_{p,q,r}=\check{\B}_{p,q,r}(\check{\Z}^2_{p,q,r}\cap\check{\Z}^3_{p,q,r})^{\times},\qquad \B^{\overline{03}}_{p,q,r}=\check{\B}_{p,q,r}\Z^{3\times}_{p,q,r}.\label{mult_B_1}
    \end{eqnarray}
\end{thm}
\begin{proof}
We have (\ref{chB_l}) because in the case of odd $n$, $W^2=0$ for any $W\in\big(\Lambd^{(0)}_{r}\big)^{\perp}_{\Z_{p,q,r}}$ by Lemma \ref{lemma_mult_c3} and in the case of even $n$, $\big(\Lambd^{(0)}_{r}\big)^{\perp}_{\Z_{p,q,r}}=\varnothing$, so we can apply Theorem \ref{lemma_AH}:
    \begin{eqnarray}
        {\B}^{\overline{12}}_{p,q,r} =\{T\in\C^{\times}_{p,q,r}:\; \widehat{\widetilde{T}}T\in\Z^{\times}_{p,q,r}\}=\check{\B}_{p,q,r}\Z^{\times}_{p,q,r},
    \end{eqnarray}
    and the proof is completed.

Let us prove (\ref{mult_B_1}).
If $n=1,2$, then  $\check{\B}^{\overline{23}}_{p,q,r}=\B^{\overline{03}}_{p,q,r}=\check{\B}_{p,q,r}=\C^{\times}_{p,q,r}$, since $\widehat{\widetilde{T}}T\in\C^{\overline{03}\times}_{p,q,r}=\C^{0\times}$ for any $T\in\C^{\times}_{p,q,r}$ by (\ref{norms_01_03}). 
In the case $n\geq 3$, we have $(W-\langle W\rangle_{\Lambda^{(0)}_r})^2=0$ for any $W\in\check{\Z}^2_{p,q,r}\cap\check{\Z}^3_{p,q,r}$ by Lemma \ref{lemma_mult_c3} and $\Lambda^{(0)}_r\subseteq(\check{\Z}^2_{p,q,r}\cap\check{\Z}^3_{p,q,r})\subseteq\C^0\oplus\rad\C_{p,q,r}$. Thus, applying Theorem \ref{lemma_AH}, we get
\begin{eqnarray*}
\check{\B}^{\overline{23}}_{p,q,r} = \{T\in\C^{\times}_{p,q,r}:\quad \widehat{\widetilde{T}}T\in(\check{\Z}^2_{p,q,r}\cap\check{\Z}^3_{p,q,r})^{\times}\}=\check{\B}_{p,q,r}(\check{\Z}^2_{p,q,r}\cap\check{\Z}^3_{p,q,r}),
\end{eqnarray*}
and the proof of the first equality in (\ref{mult_B_1}) is completed.
If  $n=3$, then $\B^{\overline{03}}_{p,q,r}=\Z^{3\times}_{p,q,r}=\C^{\times}_{p,q,r}$ (\ref{cc_3}). 
In the case $n=4$ and $n\geq 6$, $(W-\langle W\rangle_{\Lambda^{(0)}_r})^2=0$ for any $W\in{\Z}^{3}_{p,q,r}$ by Lemma \ref{lemma_mult_c3} and $\Lambda^{(0)}_r\subseteq\Z^{3}_{p,q,r}\subseteq\C^0\oplus\rad\C_{p,q,r}$, so, by Theorem \ref{lemma_AH},
\begin{eqnarray*}
{\B}^{\overline{03}}_{p,q,r} = \{T\in\C^{\times}_{p,q,r}:\quad \widehat{\widetilde{T}}T\in\Z^{3\times}_{p,q,r}\}=\check{\B}_{p,q,r}\Z^{3\times}_{p,q,r}.
\end{eqnarray*}

Consider the case $n=5$. Suppose $T\in\B^{\overline{03}}_{p,q,r}$; then $\widehat{\widetilde{T}}T=\alpha e +X+W$, where $\alpha\in\F$, $X\in\big(\C^0\big)^{\perp}_{\Lambd^{(0)}_{r}}$, and  $W\in\big(\Lambd^{(0)}_{r}\big)^{\perp}_{\Z^3_{p,q,r}}$. Note that $\alpha \neq 0$ by Lemma \ref{lemma_inv_rad}, since $\widetilde{T}T\in(\C^0\oplus\rad\C_{p,q,r})^{\times}$. 
We need to show that there exists such $Y\in\Z^{3\times}_{p,q,r}$ that $TY^{-1}\in\check{\B}_{p,q,r}$. 
Consider $Y=e+\frac{1}{2\alpha}W\in\Z^{3\times}_{p,q,r}$. 
We have $Y^{-1}=e-\frac{1}{2\alpha}W+\frac{1}{4\alpha^2}W^2$, since $W^3=0$ (see (\ref{cc_3})). We get
\begin{eqnarray}
    \widehat{\widetilde{(TY^{-1})}}(TY^{-1})=\alpha e+X-\frac{1}{4\alpha}W^2-\frac{1}{2\alpha}(WX+XW)\in\ker(\ad),
\end{eqnarray}
where we use $W^3=W^2 X=XW^2=WXW=0$, $W^2\in \Lambda^4_r\oplus\{\C^1_{p,q,0}\Lambda^4_r\}\subseteq\Lambda^{(0)}_r\oplus\C^5_{p,q,r}$, and $WX,XW\in\Lambda^5_r\oplus\{\C^1_{p,q,0}\Lambda^4_r\}\subseteq\C^5_{p,q,r}$, and this completes the proof.
\end{proof}

\section{Examples of the considered groups}\label{section_examples}

This section considers examples of the groups studied in Section \ref{section_ds_qt} in the particular cases of the geometric algebras $\C_{p,q,1}$ and Grassmann algebras $\C_{0,0,n}=\Lambda_n$. In these examples, we use the relations between the groups (Theorems \ref{theorem_rel_A} and \ref{theorem_rel_B}) and the explicit forms of the centralizers and twisted centralizers (Lemma \ref{cases_we_use}). Note that the particular case of the non-degenerate geometric algebras $\C_{p,q,0}$ is considered in Section \ref{section_ds_qt} (see Remark \ref{remark_tildeQ}, the formulas (\ref{AB_pq0_1})--(\ref{AB_pq0_2}), (\ref{chAB_pq0_1})--(\ref{chAB_pq0_2}), and the notes above them).

\begin{ex} In this example, we consider the groups in the special case of the degenerate geometric algebras $\C_{p,q,1}$ (with $r=1$). An important particular case of the algebras of this type is plane-based geometric algebra (PGA) $\C_{p,0,1}$, which has various applications in computer graphics and vision, robotics, motion capture, dynamics simulation, etc. \cite{pga}. The algebra $\C_{3,0,1}$, which is also called the Galilei--Clifford algebra, is important for the consideration of the Galilei group related to spin groups ~\cite{Abl,br1}.
In any $\C_{p,q,1}$, the groups of the type $\B$ have the form:
    \begin{eqnarray*}
    \check{\B}^{\overline{01}}_{p,q,1} &=& \{T\in\C^\times_{p,q,1}:\quad \widehat{\widetilde{T}}T\in\C^{0\times}\},
    \\
    \check{\B}^{\overline{23}}_{p,q,1} &=& \left\lbrace
                \begin{array}{lll}
                \C^{\times}_{p,q,1}, && n=1,2,
                \\
                \check{\B}^{\overline{01}}_{p,q,1}, &&n\geq 3;
                \end{array}
                \right.
    \\
    \B^{\overline{12}}_{p,q,1}&=&
    \left\lbrace
                \begin{array}{lll}
                \check{\B}^{\overline{01}}_{p,q,1}(\C^0\oplus\C^n_{p,q,1})^{\times},&& n=3\mod{4},
                \\
                \check{\B}^{\overline{01}}_{p,q,1}, && n=0,1,2\mod{4};
            \end{array}
    \right.
    \\
    \B^{\overline{03}}_{p,q,1} &=& 
     \left\lbrace
                \begin{array}{lll}
                \B^{\overline{12}}_{p,q,1}(\C^0\oplus\C^{n}_{p,q,1})^{\times},&&\mbox{$n$ is odd},
                \\
                \B^{\overline{12}}_{p,q,1},&&\mbox{$n$ is even};
                \end{array}
                \right.
    \\
    &=&
    \left\lbrace
                \begin{array}{lll}
                \check{\B}^{\overline{01}}_{p,q,1}(\C^0\oplus\C^{n}_{p,q,1})^{\times},&&n=3\mod{4},
                \\
                \check{\B}^{\overline{01}}_{p,q,1},&&n=0,1,2\mod{4}.
                \end{array}
                \right.
    \end{eqnarray*}
   Consider the particular case of the geometric algebra $\C_{p,q,1}$ with $p+q=2$, $n=3$. An arbitrary element $T\in\check{\B}^{\overline{01}}_{p,q,1}$ has the form
    \begin{eqnarray}\label{t_form}
    T = ue+ u_1e_1 +u_2e_2+u_3e_3+u_{12}e_{12}+u_{13}e_{13}+u_{23}e_{23}+u_{123}e_{123}
    \end{eqnarray}
    with the following conditions on the coefficients $u,u_1,\ldots,u_{123}\in\F$:
    \begin{eqnarray*}
        u u_{123}+u_2u_{13}-u_1u_{23}-u_3u_{12}=0,\qquad u^2-u_1^2\eta_{11}-u_2^2\eta_{22}+u_{12}\eta_{11}\eta_{22}\neq0.
    \end{eqnarray*}

In any $\C_{p,q,1}$, the groups of the type $\A$ have the form:
    \begin{eqnarray}
        \check{\A}_{p,q,1} &=& \{T\in\C^{\times}_{p,q,1}:\quad \widetilde{T}T\in\C^{0\times}\},
        \\
        \check{\A}^{\overline{12}}_{p,q,1} &=& 
        \left\lbrace
                \begin{array}{lll}
                \C^{\times}_{0,0,1},&& n=1,
                \\
                \check{\A}_{p,q,1}(\C^0\oplus\Lambda^1_1)^{\times},&&n=2,
                \\
                \check{\A}_{p,q,1},&& n\geq3;
                \end{array}
                \right.
        \\
        \check{\A}^{\overline{03}}_{p,q,1} &=& 
        \left\lbrace
                \begin{array}{lll}
                \check{\A}_{p,q,1}(\C^0\oplus\C^{2}_{p,q,1})^{\times}, &&n=2,3,
                \\
                \check{\A}_{p,q,1}, &&n\neq 2,3;
                \end{array}
                \right.
        \\
        \A^{\overline{01}}_{p,q,1} &=& 
        \left\lbrace
                \begin{array}{lll}
                \check{\A}_{p,q,1}(\C^0\oplus\C^n_{p,q,1})^{\times},&&n=1\mod{4},
                \\
                \check{\A}_{p,q,1}, &&n=0,2,3\mod{4};
                \end{array}
                \right.
        \\
        \A^{\overline{23}}_{p,q,1} &=& 
        \left\lbrace
                \begin{array}{lll}
                \A^{\overline{01}}_{p,q,1}(\C^0\oplus\Lambda_1^1)^{\times},&&n=2,
                \\
                \A^{\overline{01}}_{p,q,1}(\C^0\oplus\Lambda_1^1\oplus\C^3_{p,q,1})^{\times},&& n=3,
                \\
                \A^{\overline{01}}_{p,q,1}(\C^0\oplus\C^{n}_{p,q,1})^{\times}, &&\mbox{$n\neq3$ is odd},
                \\
                \A^{\overline{01}}_{p,q,1},&& \mbox{$n\neq2$ is even}.
                \end{array}
                \right.
    \end{eqnarray}
    Consider the particular case of $\C_{p,q,1}$ with $p+q=2$, $n=3$. An arbitrary element $T\in\check{\A}_{p,q,1}$ of the form (\ref{t_form}) has the following conditions on the coefficients $u,u_1,\ldots,u_{123}\in\F$:
    \begin{eqnarray*}
    &u u_1 - u_{2}u_{12}\eta_{22}=0,\qquad u u_2+u_1u_{12}\eta_{11}=0,
    \\
    &uu_{3}+u_1u_{13}\eta_{11}+u_{2}u_{23}\eta_{22}=0,\qquad
    u^2+u_1^2\eta_{11}+u_2^2\eta_{22}+u_{12}^2\eta_{11}\eta_{22}\neq0.
    \end{eqnarray*}
    
\end{ex}

\begin{ex}
In this example, we consider the groups in the special case of the Grassmann algebra $\C_{0,0,n}=\Lambda_n$. We have
\begin{align}
    \check{\B}^{\overline{01}}_{0,0,n} &= \{T\in\Lambda^{\times}_n:\quad \widehat{\widetilde{T}}T\in\Lambda^{(0)\times}_n\},
    \\
    \check{\B}^{\overline{23}}_{0,0,n} &=
    \left\lbrace
    \begin{array}{lll}
    \check{\B}^{\overline{01}}_{0,0,n}(\Lambda^{(0)}_n\oplus\Lambda^{n}_n)^\times,&&\mbox{$n$ is odd},
    \\
    \check{\B}^{\overline{01}}_{0,0,n}(\Lambda^{(0)}_n\oplus\Lambda^{n-1}_n)^\times,&&\mbox{$n$ is even};
    \end{array}
    \right.
    \\
    \B^{\overline{12}}_{0,0,n}&=
    \left\lbrace
    \begin{array}{lll}
    {\check{\B}}^{\overline{01}}_{0,0,n}(\Lambda^{(0)}_n\oplus\Lambda^n_n)^{\times},&&n=3\mod{4},
    \\
     {\check{\B}}^{\overline{01}}_{0,0,n},&&n=0,1,2\mod{4};
    \end{array}
    \right.
    \\
    \B^{\overline{03}}_{0,0,n}&=
    \left\lbrace
    \begin{array}{lll}
    {\check{\B}}^{\overline{01}}_{0,0,n}(\Lambda^{(0)}_n\oplus\Lambda^{n-2}_n\oplus\Lambda^n_n)^{\times},&&\mbox{$n$ is odd},
    \\
    {\check{\B}}^{\overline{01}}_{0,0,n}(\Lambda^{(0)}_n\oplus\Lambda^{n-1}_n)^{\times},&&\mbox{$n$ is even};
    \end{array}
    \right.
\\
    \check{\A}_{0,0,n}  &=\check{\A}^{\overline{03}}_{0,0,n} = \{T\in\Lambda^{\times}_n:\quad {\widetilde{T}}T\in\Lambda^{(0)\times}_n\},
    \\
    \check{\A}^{\overline{12}}_{0,0,n} &=
    \left\lbrace
    \begin{array}{lll}
    \check{\A}_{0,0,n}(\Lambda^{(0)}_n\oplus\Lambda^n_n)^{\times},&&\mbox{$n$ is odd},
    \\
    \check{\A}_{0,0,n}(\Lambda^{(0)}_n\oplus\Lambda^{n-1}_n)^{\times},&&\mbox{$n$ is even};
    \end{array}
    \right.
    \\
    \A^{\overline{01}}_{0,0,n} &= 
    \left\lbrace
    \begin{array}{lll}
    \check{\A}_{0,0,n}(\Lambda^{(0)}_n\oplus\Lambda^n_n)^{\times},&&n=1\mod{4},
    \\
    \check{\A}_{0,0,n},&&n=0,2,3\mod{4};
    \end{array}
    \right.
    \\
    \A^{\overline{23}}_{0,0,n} &=\left\lbrace
    \begin{array}{lll}
    \check{\A}_{0,0,n}(\Lambda^{(0)}_n\oplus\Lambda^{n-2}_n\oplus\Lambda^n_n)^{\times},&&\mbox{$n$ is odd},
    \\
    \check{\A}_{0,0,n}(\Lambda^{(0)}_n\oplus\Lambda^{n-1}_n)^{\times},&&\mbox{$n$ is even};
    \end{array}
    \right.
    \\
    \tilde{\Q}^{\overline{01}}_{0,0,n} &= \tilde{\Q}^{\overline{23}}_{0,0,n} =\tilde{\Q}^{\overline{12}}_{0,0,n}=\tilde{\Q}^{\overline{03}}_{0,0,n} = \Lambda^{\times}_n.
\end{align}
Let us consider the cases of the low-dimensional ($n\leq3$) Grassmann algebras $\C_{0,0,n}$. In the particular case $\C_{0,0,1}$ ($n=1$),
\begin{eqnarray}
    &&\check{\A}_{0,0,1}=\check{\A}^{\overline{03}}_{0,0,1} = \C^{0\times};
    \\
    &&\check{\B}^{\overline{01}}_{0,0,1}=\check{\B}^{\overline{23}}_{0,0,1} = \B^{\overline{12}}_{0,0,1} =  \B^{\overline{03}}_{0,0,1} =\check{\A}^{\overline{12}}_{0,0,1}=\A^{\overline{01}}_{0,0,1}=\A^{\overline{23}}_{0,0,1} 
    \\
    &&=\tilde{\Q}^{\overline{01}}_{0,0,1}= \tilde{\Q}^{\overline{23}}_{0,0,1} =\tilde{\Q}^{\overline{12}}_{0,0,1}=\tilde{\Q}^{\overline{03}}_{0,0,1} =\Lambda^{\times}_1.
\end{eqnarray}
In the case $\C_{0,0,2}$ ($n=2$),
\begin{eqnarray}
    &&\check{\A}_{0,0,2}=\check{\A}^{\overline{03}}_{0,0,2}=\A^{\overline{01}}_{0,0,2}=\Lambda^{(0)\times}_2;
    \\
    &&\check{\B}^{\overline{01}}_{0,0,2}=\check{\B}^{\overline{23}}_{0,0,2} = \B^{\overline{12}}_{0,0,2} =  \B^{\overline{03}}_{0,0,2} =\check{\A}^{\overline{12}}_{0,0,2}=\A^{\overline{23}}_{0,0,2}=\B^{\overline{03}}_{0,0,3}
    \\
    &&=\tilde{\Q}^{\overline{01}}_{0,0,2}= \tilde{\Q}^{\overline{23}}_{0,0,2} =\tilde{\Q}^{\overline{12}}_{0,0,2}=\tilde{\Q}^{\overline{03}}_{0,0,2} =\Lambda^{\times}_2.
\end{eqnarray}
In the case $\C_{0,0,3}$ ($n=3$), we have
\begin{eqnarray}
    &&\check{\A}_{0,0,3} =\check{\A}^{\overline{03}}_{0,0,3}=\check{\A}^{\overline{12}}_{0,0,3}=\A^{\overline{01}}_{0,0,3}= \Lambda^{023\times}_3;
    \\
    &&\A^{\overline{23}}_{0,0,3}=\tilde{\Q}^{\overline{01}}_{0,0,3}= \tilde{\Q}^{\overline{23}}_{0,0,3} =\tilde{\Q}^{\overline{12}}_{0,0,3}=\tilde{\Q}^{\overline{03}}_{0,0,3} =\Lambda^{\times}_3.
\end{eqnarray}
An arbitrary element $T\in\check{\B}^{\overline{01}}_{0,0,3}$ of the form
\begin{eqnarray}\label{t_form_2}
    T = ue+ u_1e_1 +u_2e_2+u_3e_3+u_{12}e_{12}+u_{13}e_{13}+u_{23}e_{23}+u_{123}e_{123}
\end{eqnarray}
has the following conditions on the coefficients $u,u_1,\ldots,u_{123}\in\F$:
\begin{eqnarray}
    u\neq 0,\qquad uu_{123}-u_1u_{23}+u_2u_{13}-u_3u_{12}=0.
\end{eqnarray}
\end{ex}

\section{Conclusions}

In this paper, we consider the groups $\Gamm^{\overline{kl}}_{p,q,r}$, $\check{\Gamm}^{\overline{kl}}_{p,q,r}$, and $\tilde{\Gamm}^{\overline{kl}}_{p,q,r}$ preserving the direct sums $\C^{\overline{kl}}_{p,q,r}$, $k,l=0,1,2,3$, of the subspaces determined by the grade involution and reversion under the adjoint representation $\ad$ and twisted adjoint representations $\check{\ad}$ and $\tilde{\ad}$. 
We prove that these groups are defined using the norm functions $\psi$ and $\chi$ and the centralizers $\Z^m_{p,q,r}$ and twisted centralizers $\check{\Z}^m_{p,q,r}$ of the subspaces of fixed grades.
In Table \ref{table_all_groups}, we present a comprehensive list of the groups under consideration in this work and their defining conditions. Namely, the first column indicates the group, while the second and third columns contain the set to which the norm functions $\psi$ and $\chi$ respectively of the group's elements belong  according to the group's definition. 
\begin{table}
\caption{Generalized Clifford and Lipschitz groups and auxiliary groups $\check{\A}_{p,q,r}$ and $\check{\B}_{p,q,r}$}\label{table_all_groups} 
\begin{tabular}{| c | c  | c |} \hline
Lie group\rule{0pt}{3.5ex}  & $\psi(T)=\widetilde{T}T$ & $\chi(T)=\widehat{\widetilde{T}}T$ \\ \hline
$\A^{\overline{01}}_{p,q,r}=\Gamm^{\overline{01}}_{p,q,r}$\rule{0pt}{3ex}    & $\Z^{1\times}_{p,q,r}=\ker(\ad)$ & \\ 
$\A^{\overline{23}}_{p,q,r}=\Gamm^{\overline{23}}_{p,q,r}$\rule{0pt}{3ex}  & $(\Z^2_{p,q,r}\cap\Z^3_{p,q,r})^\times$ & \\ 
$\B^{\overline{12}}_{p,q,r}=\Gamm^{\overline{12}}_{p,q,r}$\rule{0pt}{3ex}  & & $\Z^{1\times}_{p,q,r}=\ker(\ad)$ \\ 
$\B^{\overline{03}}_{p,q,r}=\Gamm^{\overline{03}}_{p,q,r}$\rule{0pt}{3ex}  & & $\Z^{3\times}_{p,q,r}$ \\ \hline
$\check{\A}^{\overline{12}}_{p,q,r}=\check{\Gamm}^{\overline{12}}_{p,q,r}$\rule{0pt}{3ex}  & $(\check{\Z}^1_{p,q,r}\cap\check{\Z}^2_{p,q,r})^{\times}$ & \\ 
$\check{\A}^{\overline{03}}_{p,q,r}=\check{\Gamm}^{\overline{03}}_{p,q,r}$\rule{0pt}{3ex}  & $({\Z}^3_{p,q,r}\cap\C^{(0)}_{p,q,r})^{\times}$ & \\
$\check{\B}^{\overline{01}}_{p,q,r}=\check{\Gamm}^{\overline{01}}_{p,q,r}$\rule{0pt}{3ex}  & & $({\Z}^1_{p,q,r}\cap\C^{(0)}_{p,q,r})^{\times}=\ker(\check{\ad})$ \\
$\check{\B}^{\overline{23}}_{p,q,r}=\check{\Gamm}^{\overline{23}}_{p,q,r}$\rule{0pt}{3ex}  & & $(\check{\Z}^2_{p,q,r}\cap\check{\Z}^3_{p,q,r})^{\times}$ \\ \hline 
$\tilde{\Q}^{\overline{01}}_{p,q,r}=\tilde{\Gamm}^{\overline{01}}_{p,q,r}$\rule{0pt}{3ex}  & $\Z^{4\times}_{p,q,r}$ & $\check{\Z}^{1\times}_{p,q,r}=\ker(\tilde{\ad})$ \\ 
$\tilde{\Q}^{\overline{23}}_{p,q,r}=\tilde{\Gamm}^{\overline{23}}_{p,q,r}$\rule{0pt}{3ex}  & $\Z^{2\times}_{p,q,r}$ & $\check{\Z}^{3\times}_{p,q,r}$ \\ 
$\tilde{\Q}^{\overline{12}}_{p,q,r}=\tilde{\Gamm}^{\overline{12}}_{p,q,r}$\rule{0pt}{3ex}  & $\check{\Z}^{1\times}_{p,q,r}=\ker(\tilde{\ad})$ & $\Z^{2\times}_{p,q,r}$ \\ 
$\tilde{\Q}^{\overline{03}}_{p,q,r}=\tilde{\Gamm}^{\overline{03}}_{p,q,r}$\rule{0pt}{3ex}  & $\check{\Z}^{3\times}_{p,q,r}$ & $\Z^{4\times}_{p,q,r}$ \\  \hline 
$\check{\A}_{p,q,r}$\rule{0pt}{3ex} & $\ker(\check{\ad})$ & \\
$\check{\B}_{p,q,r}$\rule{0pt}{3ex} &  & $\ker(\check{\ad})$\\ \hline
\end{tabular}
\end{table}

We prove that the groups of types $\A$ and $\B$ are related to each other through multiplication, using the auxiliary simple groups $\check{\A}_{p,q,r}$ and $\check{\B}_{p,q,r}$:
\begin{eqnarray*}
\!\!\!\!\!\!\!\!\!\!&\check{\A}^{\overline{03}}_{p,q,r}=\check{\A}_{p,q,r}({\Z}^3_{p,q,r}\cap\C^{(0)}_{p,q,r})^{\times},\quad \check{\A}^{\overline{12}}_{p,q,r}=\check{\A}_{p,q,r}(\check{\Z}^1_{p,q,r}\cap\check{\Z}^2_{p,q,r})^{\times},
\\
\!\!\!\!\!\!\!\!\!\!&\A^{\overline{23}}_{p,q,r}= \A^{\overline{01}}_{p,q,r}(\Z^2_{p,q,r}\cap\Z^3_{p,q,r})^{\times},
\quad \A^{\overline{01}}_{p,q,r} = \check{\A}_{p,q,r}\Z^{\times}_{p,q,r}.
\end{eqnarray*}
and
\begin{eqnarray*}
        \B^{\overline{12}}_{p,q,r} = \check{\B}_{p,q,r}\Z^{\times}_{p,q,r},
        \;\check{\B}^{\overline{23}}_{p,q,r}=\check{\B}_{p,q,r}(\check{\Z}^2_{p,q,r}\cap\check{\Z}^3_{p,q,r})^{\times},\;\B^{\overline{03}}_{p,q,r}=\check{\B}_{p,q,r}\Z^{3\times}_{p,q,r}.
    \end{eqnarray*}
In the particular cases of $\C_{p,q,r}$, where the centralizers in the group definitions have a similar form, many groups coincide. Specifically, for a fixed dimension $n$, the less degenerate the algebra, the more the groups of types $\Q$, $\A$, and $\B$ coincide (see Remark \ref{remark_tildeQ}  and the notes before Theorems \ref{theorem_AB} and \ref{theorem_chAB} respectively). Conversely, in the Grassmann algebras $\C_{0,0,n}$ of sufficiently large dimension, almost all the groups are distinct (see the example in Section \ref{section_examples}).

The groups presented in Table~\ref{table_all_groups} can be considered as generalized degenerate Clifford and Lipschitz groups, since they  preserve fixed subspaces under the adjoint and twisted adjoint representations. The considered groups contain standard Clifford and Lipschitz as subgroups and are useful for the study of the generalized degenerate spin groups. 
These groups can be useful for applications of geometric algebras in computer science \cite{ce,h1,cNN0}, physics \cite{lg1,phys,hestenes}, and engineering \cite{cv2,ce}. The generalized degenerate Clifford and Lipschitz groups can be applied in deep learning to construct neural networks that are equivariant with respect to the action of pseudo-orthogonal groups \cite{cNN}. 
Namely, one can construct neural network layers that are equivariant with respect to a certain generalized Lipschitz group by parameterizing mappings that are equivariant with respect to this group \cite{FS2025}. For example, these mappings include projections onto the direct sums of the subspaces determined by the grade involution and reversion. These layers will automatically be equivariant with respect to the corresponding pseudo-orthogonal groups.
Moreover, the generalized Lipschitz groups can be interesting for considering the Galilei group related to the spin groups in the Galilei--Clifford algebra $\C_{3,0,1}$~\cite{Abl,br1,br2}.
Also they might be useful for working with orthogonal transformations in PGA $\C_{p,0,1}$, which are applied in computer graphics and vision, robotics, motion capture, dynamics simulation \cite{pga}.

\section*{Acknowledgements}

The results of this paper were reported at the ENGAGE Workshop (Geneva, Switzerland, July 2024) within the International Conference Computer Graphics International 2024 (CGI 2024). The authors are grateful to the organizers and participants of this conference for fruitful discussions.

The authors are grateful to the anonymous reviewers for their careful reading of the paper and helpful comments on how to improve the presentation.

The publication was prepared within the framework of the Academic Fund Program at HSE University (grant 24-00-001 Clifford algebras and applications).

\medskip

\noindent{\bf Data availability} Data sharing is not applicable to this article as no datasets were generated or analyzed during the current study.

\medskip

\noindent{\bf Declarations}\\
\noindent{\bf Conflict of interest} The authors declare that they have no conflict of interest.

\appendix

\section{Summary of notation}\label{section_notation}

Following the reviewer’s recommendation, we include in Tables \ref{table_notation1}--\ref{table_notation2} a summary of the notation used throughout the paper. The table lists each notation, its meaning, and the place where it first appears.

\begin{table}[h]
\caption{Summary of notation (part 1)}\label{table_notation1} 
\begin{tabular}{| c | c | c |}  \hline
Notation  & Meaning & First mention \\ \hline
$\C_{p,q,r}$ & \parbox{5.2cm}{\begin{center}(Clifford) geometric algebra over the real $\BR^{p,q,r}$ or complex $\BC^{p+q,0,r}$ vector space\end{center}} & page \pageref{section_def}\\ \cline{2-2}
$\F$ & \parbox{5.2cm}{\begin{center}field of real or complex numbers in the cases $\BR^{p,q,r}$ and $\BC^{p+q,0,r}$  respectively\end{center}} & page \pageref{section_def}  \\ \cline{2-2}
$\Lambda_r$ & Grassmann subalgebra $\C_{0,0,r}$ & page \pageref{gen} \\ \cline{2-2}
$e_1,\ldots,e_n$ & generators of $\C_{p,q,r}$ & page \pageref{gen} \\ \cline{2-2}
$\rad\C_{p,q,r}$ & Jacobson radical of $\C_{p,q,r}$ & page \pageref{rad} \\ \hline
$\C^{0}$ & subspace of grade $0$ & page \pageref{gen}\\ \cline{2-2}
$\C^{k}_{p,q,r}$ & subspace of fixed grade $k=1,\ldots,n$ & page \pageref{gen} \\ \cline{2-2}
$\C^{\geq k}_{p,q,r}$, $\C^{\leq k}_{p,q,r}$ & \parbox{5.2cm}{\begin{center}direct sums of subspaces of grades larger/smaller or equal to $k$\end{center}} & \parbox{2.2cm}{\begin{center}formulas \eqref{def_geq_k}--\eqref{def_leq_k}\end{center}} \\ \cline{2-2}
$\{\C^k_{p,q,0}\Lambd^l_r\}$ & \parbox{5.2cm}{\begin{center}subspace of $\C_{p,q,r}$ spanned by the elements of the form $ab$, where $a\in \C^k_{p,q,0}$ and $b\in\Lambd^l_r$\end{center}} & page \pageref{Zpqr}\\ \hline
$\stackrel{}{\widehat{U}}$ & grade involute of $U\in\C_{p,q,r}$ & page \pageref{evenodd} \\ \cline{2-2}
$\widetilde{U}$ & reverse of $U\in\C_{p,q,r}$ & page \pageref{evenodd} \\ \cline{2-2}
$\widehat{\widetilde{U}}$ & Clifford conjugate of $U\in\C_{p,q,r}$ & page \pageref{evenodd} \\ \hline
$\C^{(0)}_{p,q,r}$, $\C^{(1)}_{p,q,r}$ & even and odd subspaces & formula \eqref{evenodd} \\ \cline{2-2}
$\langle\rangle_{(0)}$, $\langle\rangle_{(1)}$ & \parbox{5.2cm}{\begin{center}projections onto $\C^{(0)}_{p,q,r}$, $\C^{(1)}_{p,q,r}$\end{center}} & page \pageref{evenodd} \\ \cline{2-2}
$\langle\rangle_{\Lambda^{(0)}_r}$ & \parbox{5.2cm}{\begin{center}projection onto even Grassmann subalgebra $\Lambda^{(0)}_r$\end{center}} & page \pageref{def_chB} \\ \hline
$\C^{\overline{k}}_{p,q,r}$ & \parbox{5.2cm}{\begin{center}subspaces determined by grade involution and reversion\end{center}} & formula \eqref{qtdef}\\ \cline{2-2}
$\C^{\overline{kl}}_{p,q,r}$ & \parbox{5.2cm}{\begin{center}direct sums of $\C^{\overline{k}}_{p,q,r}$, $k=0,1,2,3$\end{center}} & page \pageref{qtdef} \\ \hline
$\Z_{p,q,r}$ & center of $\C_{p,q,r}$ & formula \eqref{Zpqr} \\ \cline{2-2}
$\Z^m_{p,q,r}$, $\check{\Z}^m_{p,q,r}$ & \parbox{5.2cm}{\begin{center}centralizers and twisted centralizers respectively of $\C^{m}_{p,q,r}$, $m=0,\ldots,n$\end{center}} & \parbox{2.2cm}{\begin{center}formulas \eqref{def_Zm}--\eqref{def_chZm}\end{center}} \\ \cline{2-2}
$\Z^{\overline{k}}_{p,q,r}$, $\check{\Z}^{\overline{k}}_{p,q,r} $& \parbox{5.2cm}{\begin{center}centralizers and twisted centralizers respectively of $\C^{\overline{k}}_{p,q,r}$, $k=0,1,2,3$\end{center}} & \parbox{2.2cm}{\begin{center}formulas \eqref{def_CC_ov}--\eqref{def_chCC_ov}\end{center}}\\ \cline{2-2}
${\Z}^{\overline{kl}}_{p,q,r} $, $\check{\Z}^{\overline{kl}}_{p,q,r} $& \parbox{5.2cm}{\begin{center}centralizers and twisted centralizers respectively of $\C^{\overline{kl}}_{p,q,r}$, $k,l=0,1,2,3$\end{center}} & formula \eqref{def_sumZ}\\ 
\hline
\end{tabular}
\end{table}

\begin{table}[h]
\caption{Summary of notation (part 2)}\label{table_notation2} 
\begin{tabular}{| c | c | c |}  \hline
Notation  & Meaning & First mention \\ \hline
$\H^\times$ & \parbox{5.2cm}{\begin{center}subset of all invertible elements of a set $\H$\end{center}}  & page \pageref{def_cg} \\ \hline
$\Gamma_{p,q,r}$ & Clifford group & formula \eqref{def_cg}\\ \cline{2-2}
$\Gamma^{\pm}_{p,q,r}$ & Lipschitz group & formula \eqref{def_lg}\\ \hline
$\psi$, $\chi$ & norm functions of $\C_{p,q,r}$ elements & formula \eqref{norm_functions} \\ \hline
$\ad$ & \parbox{5.2cm}{\begin{center}adjoint representation ${\ad}_{T}(U)=TU T^{-1}$\end{center}} &  formula (\ref{ar})\\ \cline{2-2}
$\check{\ad}$ & \parbox{5.2cm}{\begin{center}twisted adjoint representation $\check{\ad}_{T}(U)=\widehat{T}U T^{-1}$\end{center}} & \parbox{2.2cm}{\begin{center}formula (\ref{twa1})  \end{center}} \\ \cline{2-2}
$\tilde{\ad}$ & \parbox{5.2cm}{\begin{center}twisted adjoint representation $\tilde{\ad}_{T}(U)=TU_0 T^{-1}+\widehat{T} U_1 T^{-1}$\end{center}} & \parbox{2.2cm}{\begin{center}formula (\ref{twa22})\end{center}} \\ \cline{2-2}
\parbox{3.2cm}{\begin{center}$\ker(\ad)$, $\ker{(\check{\ad})}$, $\ker{(\tilde{\ad})}$\end{center}} & kernels of $\ad$, $\check{\ad}$, and $\tilde{\ad}$ & page \pageref{ker_ad}\\ \hline
\parbox{3.2cm}{\begin{center}$\Gamm^{\overline{kl}}_{p,q,r}$,
$\check{\Gamm}^{\overline{kl}}_{p,q,r}$,
$\tilde{\Gamm}^{\overline{kl}}_{p,q,r}$\end{center}} & \parbox{5.2cm}{\begin{center}groups preserving $\C^{\overline{kl}}_{p,q,r}$ under $\ad$, $\check{\ad}$, and $\tilde{\ad}$\end{center}} & \parbox{2.2cm}{\begin{center}formulas \eqref{def_gammakl}--\eqref{ti_Gamma_kl}\end{center}}\\ \cline{2-2}
\parbox{2.2cm}{\begin{center}$\A^{\overline{01}}_{p,q,r}$, $\B^{\overline{12}}_{p,q,r}$, $\A^{\overline{23}}_{p,q,r}$, $\B^{\overline{03}}_{p,q,r}$\end{center}}  & \parbox{4.2cm}{\begin{center}groups preserving $\C^{\overline{kl}}_{p,q,r}$ under $\ad$\end{center}} & \parbox{2.2cm}{\begin{center}formulas \eqref{def_A01}--\eqref{def_B03}\end{center}}\\ \cline{2-2}
\parbox{2.2cm}{\begin{center}$\check{\A}^{\overline{12}}_{p,q,r}$, $\check{\B}^{\overline{01}}_{p,q,r}$, $\check{\A}^{\overline{03}}_{p,q,r}$, $\check{\B}^{\overline{23}}_{p,q,r}$ \end{center}} & \parbox{4.2cm}{\begin{center}groups preserving $\C^{\overline{kl}}_{p,q,r}$ under $\check{\ad}$\end{center}} & \parbox{2.2cm}{\begin{center}formulas \eqref{def_chA12}--\eqref{def_chB23_}\end{center}}\\ \cline{2-2}
\parbox{2.2cm}{\begin{center}$\tilde{\Q}^{\overline{01}}_{p,q,r}$, $\tilde{\Q}^{\overline{23}}_{p,q,r}$, $\tilde{\Q}^{\overline{12}}_{p,q,r}$, $\tilde{\Q}^{\overline{03}}_{p,q,r}$\end{center}} & \parbox{4.2cm}{\begin{center}groups preserving $\C^{\overline{kl}}_{p,q,r}$ under $\tilde{\ad}$\end{center}} & \parbox{2.2cm}{\begin{center}formulas \eqref{def_tQ01}--\eqref{def_tQ03_} \end{center}}\\ \cline{2-2}
$\check{\A}_{p,q,r}$, $\check{\B}_{p,q,r}$ & auxiliary groups & \parbox{2.2cm}{\begin{center}formulas \eqref{def_chA}, \eqref{def_chB}\end{center}} \\ \hline 
$\D^{\perp}_{\H}$ & \parbox{5.2cm}{\begin{center}orthogonal complement of a subset $\D\subseteq \H$\end{center}} & formula \eqref{def_orthcomp} \\ \hline
\end{tabular}
\end{table}

\section{The corresponding Lie algebras}\label{section_algebras}

\begin{thm}\label{thm_alg}
    The Lie algebras of the Lie groups of type $\A$ ($\A^{\overline{01}}_{p,q,r}$, $\A^{\overline{23}}_{p,q,r}$, $\check{\A}^{\overline{12}}_{p,q,r}$, and $\check{\A}^{\overline{03}}_{p,q,r}$),  type B ($\B^{\overline{12}}_{p,q,r}$, $\B^{\overline{03}}_{p,q,r}$, $\check{\B}^{\overline{01}}_{p,q,r}$, and $\check{\B}^{\overline{23}}_{p,q,r}$), and type $\Q$ ($\tilde{\Q}^{\overline{01}}_{p,q,r}$, $\tilde{\Q}^{\overline{23}}_{p,q,r}$, $\tilde{\Q}^{\overline{12}}_{p,q,r}$,  and $\tilde{\Q}^{\overline{03}}_{p,q,r}$) and their dimensions are presented in Tables~\ref{table_lie_algebras_A}--\ref{table_lie_algebras_Q2}.
\end{thm}
\begin{proof}
We use the well-known facts about the relation between an arbitrary Lie group and the corresponding Lie
algebra to prove the statements. The dimensions of these Lie algebras are calculated using (see, for example, \cite{lie_alg})
\begin{eqnarray}
\!\!\!\!\!\!\!\!\!\!\!\!\!&\dim \C^{k}_{p,q,r} = \left\lbrace
            \begin{array}{lll}
        \binom{n}{k},&& \!\!\!\!\!n\geq k,
        \\
        0,&& \!\!\!\!\!n< k;
        \end{array}
            \right.
\\
\!\!\!\!\!\!\!\!\!\!\!\!\!&\dim\C^{\overline{0}}_{p,q,r} = 2^{n-2}+2^{\frac{n}{2}-1}\cos(\frac{\pi n}{4}),\; \dim\C^{\overline{1}}_{p,q,r} = 2^{n-2}+2^{\frac{n}{2}-1}\sin(\frac{\pi n}{4}),
\\
\!\!\!\!\!\!\!\!\!\!\!\!\!&\dim\C^{\overline{2}}_{p,q,r} = 2^{n-2}-2^{\frac{n}{2}-1}\cos(\frac{\pi n}{4}),\;\dim\C^{\overline{3}}_{p,q,r} = 2^{n-2}-2^{\frac{n}{2}-1}\sin(\frac{\pi n}{4}),
\end{eqnarray}
where $\binom{n}{k}=\frac{n!}{k!(n-k)!}$ is the binomial coefficient.
\end{proof}

In Tables \ref{table_lie_algebras_A}--\ref{table_lie_algebras_Q2}, we write out the Lie algebras of the Lie groups considered in Section \ref{section_ds_qt}. The first column of each table contains the Lie groups, while the fourth and fifth columns provide the corresponding Lie algebras and dimensions respectively. 
The sets in the fourth column are considered with respect to the commutator $[U,V]=UV-VU$.
The second and third 
columns of the tables contain conditions on $n$ and $r$ respectively. They are empty for the Lie groups that have the same Lie algebra for any natural $n\geq1$ and $r\geq0$ respectively. 

Denote by $\dim_{\A}$, $\dim_{\B}$, and $\dim_{\Q}$ the following values:
\begin{eqnarray*}
\dim_{\A}\!\!\!\!\!\! &:=& \!\!\!\!\!\!2^{n-1}-2^{\frac{n}{2}-1}(\cos(\frac{\pi n}{4})+\sin(\frac{\pi n}{4}))\!+\!
\left\lbrace
            \begin{array}{lll}
       \!\!\! 2^{r-2}+2^{\frac{r}{2}-1}\cos(\frac{\pi r}{4}),&& \!\!\!\!\!\!\!\! r\geq 1,
        \\
       \!\!\! 1,&&\!\!\! \!\!\!\!\! r=0;
        \end{array}
            \right.
    \\
\dim_{\B}\!\!\!\!\!\! &:=&\!\!\!\!\!\!  2^{n-1}+2^{\frac{n}{2}-1}(\sin(\frac{\pi n}{4})-\cos(\frac{\pi n}{4}))\!+\!
\left\lbrace
            \begin{array}{lll}
            \!\!\!2^{r-2}+2^{\frac{r}{2}-1}\cos(\frac{\pi r}{4}),&& \!\!\!\!\!\!\!\! r\geq 1,
            \\
            \!\!\!1,&&\!\!\! \!\!\!\!\! r=0;
            \end{array}
            \right.
    \\
  \dim_{\Q} \!\!\!\!\!\!&:=&  \!\!\!\!\!\!2^{n-2}- 2^{\frac{n}{2}-1}\cos(\frac{\pi n}{4})\! +\!
\left\lbrace
            \begin{array}{lll}
\!\!\!3\cdot 2^{r-2} + 2^{\frac{r}{2}-1}(\sin(\frac{\pi r}{4})+ \cos(\frac{\pi r}{4})),&& \!\!\!\!\!\!\!\! r\geq 1,
\\
\!\!\!1,&&\!\!\! \!\!\!\!\! r=0.
            \end{array}
            \right.
\end{eqnarray*}

\begin{landscape}
\begin{table}
\caption{Lie groups of $\A$-type and the corresponding Lie algebras}\label{table_lie_algebras_A} 
\begin{tabular}{| c | c | c | c | c |}  \hline
Lie group  & $n\mod{4}$ & $r$ & Lie algebra & Dimension\\ \hline
$\A^{\overline{01}}_{p,q,r}$\rule{0pt}{3ex} & $0,2,3$ & & $\Lambda^{\overline{0}}_r\oplus\C^{\overline{23}}_{p,q,r}$ & $\dim_{\A}$\\
$\A^{\overline{23}}_{p,q,r}$\rule{0pt}{3ex} & $0$ & & & \\ 
& $2$ & $r\leq n-2$ &  & \\
& $3$ & $r\leq n-3$ &  & \\
$\check{\A}^{\overline{12}}_{p,q,r}$\rule{0pt}{3ex} & $0,3$ & & &  \\
& $1$ & $r\neq n$&  & \\
 & $2$ & $r\leq n-2$ &  & \\
$\check{\A}^{\overline{03}}_{p,q,r}$\rule{0pt}{3ex} & $2,3$ & &  &  \\
 & $0$ & $r\leq n-3,r=n$ &  & \\
  & $1$ & $r\leq n-4,r=n$ &  & \\ \hline
$\A^{\overline{01}}_{p,q,r}$, $\A^{\overline{23}}_{p,q,r}$\rule{0pt}{3ex} & $1$ & & $\Lambda^{\overline{0}}_r\oplus\C^{\overline{23}}_{p,q,r}\oplus\C^n_{p,q,r}$  & $\dim_{\A}+1$ \\
 $\check{\A}^{\overline{12}}_{p,q,r}$\rule{0pt}{3ex} & $1$ & $r=n$ &  & \\ 
 $\check{\A}^{\overline{03}}_{p,q,r}$\rule{0pt}{3ex} & $0$ & $r=n-2,n-1$ &  & \\  \hline 
$\A^{\overline{23}}_{p,q,r}$, $\check{\A}^{\overline{12}}_{p,q,r}$\rule{0pt}{3ex} & $2$ & $r=n-1,n$ & $\Lambda^{\overline{0}}_r\oplus\C^{\overline{23}}_{p,q,r}\oplus\Lambda^{n-1}_r$ &  $\dim_{\A} + \frac{r!}{(n-1)!(r-n+1)!}$\\  \hline
$\A^{\overline{23}}_{p,q,r}$\rule{0pt}{3ex} & $3$ & $r\geq n-2$ & $\Lambda^{\overline{0}}_r\oplus\C^{\overline{23}}_{p,q,r}\oplus\Lambda^{n-2}_r$ & $\dim_{\A}+\frac{r!}{(n-2)!(r-n+2)!}$ \\ \hline
$\check{\A}^{\overline{03}}_{p,q,r}$\rule{0pt}{3ex} & $1$ & $r=n-1$ & $\Lambda^{\overline{0}}_r\oplus\C^{\overline{23}}_{p,q,r}\oplus\{\C^1_{p,q,0}\Lambda^{n-2}_r\}$  & $\dim_{\A}+n-1$ \\ \hline
$\check{\A}^{\overline{03}}_{p,q,r}$\rule{0pt}{3ex} & $1$ & $r=n-3$ & $\Lambda^{\overline{0}}_r\oplus\C^{\overline{23}}_{p,q,r}\oplus\{\C^2_{p,q,0}\Lambda^{n-3}_r\}$  & $\dim_{\A}+3$\\ \hline
$\check{\A}^{\overline{03}}_{p,q,r}$\rule{0pt}{3ex} & $1$ & $r=n-2$ & $\Lambda^{\overline{0}}_r\oplus\C^{\overline{23}}_{p,q,r}\oplus\{\C^1_{p,q,0}\Lambda^{n-2}_r\}\oplus\{\C^2_{p,q,0}\Lambda^{n-3}_r\}$  & $\dim_{\A}+n$\\ \hline
\end{tabular}
\end{table}
\end{landscape}

\begin{landscape}
\begin{table}
\caption{Lie groups of $\B$-type and the corresponding Lie algebras}\label{table_lie_algebras_B} 
\begin{tabular}{| c | c | c | c | c |}  \hline
Lie group  & $n\mod{4}$ & $r$ & Lie algebra & Dimension\\ \hline
$\check{\B}^{\overline{01}}_{p,q,r}$\rule{0pt}{3ex} & & & $\Lambda^{\overline{0}}_r\oplus\C^{\overline{12}}_{p,q,r}$ & $\dim_{\B}$ \\ 
$\B^{\overline{12}}_{p,q,r}$\rule{0pt}{3ex} & $0,1,2$ & & & \\ 
$\B^{\overline{03}}_{p,q,r}$\rule{0pt}{3ex} & $0$ & $r\leq n-3$ & & \\
& $1$ & $r\leq n-4$ & & \\ 
& $2$ & & & \\  
$\check{\B}^{\overline{23}}_{p,q,r}$\rule{0pt}{3ex} & $1,2$ & & & \\
& $0$ & $r\leq n-3$ & & \\ 
& $3$ & $r\leq n-2$ & & \\ \hline
$\B^{\overline{03}}_{p,q,r}$, $\B^{\overline{12}}_{p,q,r}$\rule{0pt}{3ex} & $3$ & & $\Lambda^{\overline{0}}_r\oplus\C^{\overline{12}}_{p,q,r}\oplus\C^{n}_{p,q,r}$ & $\dim_{\B}+1$\\
$\check{\B}^{\overline{23}}_{p,q,r}$ & $3$ & $r=n-1,n$ && \\ \hline
$\B^{\overline{03}}_{p,q,r}$, $\check{\B}^{\overline{23}}_{p,q,r}$ \rule{0pt}{3ex}& $0$ & $r=n$ & $\Lambda^{\overline{0}}_r\oplus\C^{\overline{12}}_{p,q,r}\oplus\Lambda^{n-1}_r$ & $\dim_{\B}+n$ \\ \hline
$\B^{\overline{03}}_{p,q,r}$\rule{0pt}{3ex} & $1$ & $r=n$ & $\Lambda^{\overline{0}}_r\oplus\C^{\overline{12}}_{p,q,r}\oplus\Lambda^{n-2}_r$ & $\dim_{\B}+\frac{n(n-1)}{2}$ \\ \hline
$\B^{\overline{03}}_{p,q,r}$, $\check{\B}^{\overline{23}}_{p,q,r}$\rule{0pt}{3ex}& $0$& $r=n-2$ & $\Lambda^{\overline{0}}_r\oplus\C^{\overline{12}}_{p,q,r}\oplus\{\C^{1}_{p,q,0}\Lambda^{n-2}_r\}\oplus\{\C^2_{p,q,0}\Lambda^{n-2}_r\}$ & $\dim_{\B}+3$ \\ \hline
$\B^{\overline{03}}_{p,q,r}$\rule{0pt}{3ex} & $1$ & $r=n-3$ & $\Lambda^{\overline{0}}_r\oplus\C^{\overline{12}}_{p,q,r}\oplus\{\C^{1}_{p,q,0}\Lambda^{n-3}_r\}\oplus\{\C^2_{p,q,0}\Lambda^{n-3}_r\}$ & $\dim_{\B}+6$ \\ \hline
$\B^{\overline{03}}_{p,q,r}$, $\check{\B}^{\overline{23}}_{p,q,r}$\rule{0pt}{3ex} & $0$  & $r=n-1$ & $\Lambda^{\overline{0}}_r\oplus\C^{\overline{12}}_{p,q,r}\oplus\Lambda^{n-1}_r\oplus\{\C^{1}_{p,q,0}\Lambda^{n-2}_r\}\oplus\{\C^1_{p,q,0}\Lambda^{n-1}_r\}$ & $\dim_{\B}+n+1$ \\ \hline
$\B^{\overline{03}}_{p,q,r}$\rule{0pt}{3ex} & $1$& $r=n-1$ & $\Lambda^{\overline{0}}_r\oplus\C^{\overline{12}}_{p,q,r}\oplus\Lambda^{n-2}_r\oplus\{\C^{1}_{p,q,0}\Lambda^{n-3}_r\}\oplus\{\C^1_{p,q,0}\Lambda^{n-2}_r\}$ & $\dim_{\B}+\frac{(n-1)(n+2)}{2}$ \\ \hline
$\B^{\overline{03}}_{p,q,r}$\rule{0pt}{3ex} & $1$ & $r=n-2$ & $\Lambda^{\overline{0}}_r\oplus\C^{\overline{12}}_{p,q,r}\oplus\Lambda^{n-2}_r\oplus\{\C^{1}_{p,q,0}\Lambda^{n-3}_r\}$ & $\dim_{\B}+3n-3$ \\ 
& & & $\oplus\{\C^1_{p,q,0}\Lambda^{n-2}_r\}\oplus\{\C^2_{p,q,0}\Lambda^{n-3}_r\}$& \\ \hline
\end{tabular}
\end{table}
\end{landscape}

\begin{landscape}
\begin{table}
\caption{Lie groups of $\Q$-type and the corresponding Lie algebras (part 1)}\label{table_lie_algebras_Q1} 
\begin{tabular}{| c | c | c | c | c |}  \hline
Lie group  & $n\mod{4}$ & $r$ & Lie algebra & Dimension\\ \hline
$\tilde{\Q}^{\overline{01}}_{p,q,r}$\rule{0pt}{3ex} & $0$ & & $\Lambda^{\overline{013}}_{r}\oplus\C^{\overline{2}}_{p,q,r}$ & $\dim_{\Q}$ \\
& $1$ & $r=n$ & & \\
& $2$ & $r=n$, $r\leq n-4$ & & \\
& $3$ & $r=n$, $r\leq n-5$ & & \\ 
$\tilde{\Q}^{\overline{23}}_{p,q,r}$\rule{0pt}{3ex} & $0$ & $r=n$, $r\leq n-4$ & & \\ 
& $1$ & $r=n$ & & \\ 
& $2$ & & & \\ 
& $3$ & $r=n$, $r\leq n-3$ & & \\ 
$\tilde{\Q}^{\overline{12}}_{p,q,r}$\rule{0pt}{3ex} & $0,1,2$ & & & \\
& $3$ & $r=n$ & & \\ 
$\tilde{\Q}^{\overline{03}}_{p,q,r}$\rule{0pt}{3ex}  & $0$ & $r=n$, $r\leq n-4$ & & \\ 
& $1$ & $r=n$, $r\leq n-5$ &  & \\
& $2$ & $r=n$, $r\leq n-4$ & & \\ 
& $3$ & $r=n$ & &  \\ \hline
$\tilde{\Q}^{\overline{01}}_{p,q,r}$, $\tilde{\Q}^{\overline{23}}_{p,q,r}$\rule{0pt}{3ex} & $1$& $r\neq n$& $\Lambda^{\overline{013}}_{r}\oplus\C^{\overline{2}}_{p,q,r}\oplus\C^n_{p,q,r}$ & $\dim_{\Q}+1$ \\ 
$\tilde{\Q}^{\overline{23}}_{p,q,r}$ & $3$ & $r=n-2,n-1$ & & \\ 
$\tilde{\Q}^{\overline{12}}_{p,q,r}$, $\tilde{\Q}^{\overline{03}}_{p,q,r}$\rule{0pt}{3ex} & $3$ & $r\neq n$ & & \\ \hline
$\tilde{\Q}^{\overline{01}}_{p,q,r}$, $\tilde{\Q}^{\overline{03}}_{p,q,r}$\rule{0pt}{3ex} & $2$ & $r=n-1$ & $\Lambda^{\overline{013}}_{r}\oplus\C^{\overline{2}}_{p,q,r}\oplus\{\C^1_{p,q,0}\Lambda^{n-2}_r\}$ & $\dim_{\Q}+n-1$ \\ \hline
$\tilde{\Q}^{\overline{23}}_{p,q,r}$, $\tilde{\Q}^{\overline{03}}_{p,q,r}$\rule{0pt}{3ex}  & $0$ & $r=n-1$ & $\Lambda^{\overline{013}}_{r}\oplus\C^{\overline{2}}_{p,q,r}\oplus\{\C^1_{p,q,0}\Lambda^{n-2}_r\}\oplus\C^n_{p,q,r}$  & $\dim_{\Q}+n$ \\ \hline
$\tilde{\Q}^{\overline{01}}_{p,q,r}$\rule{0pt}{3ex} & $3$& $r=n-1$& $\Lambda^{\overline{013}}_{r}\oplus\C^{\overline{2}}_{p,q,r}\oplus\{\C^1_{p,q,0}\Lambda^{n-3}_r\}$ & $\dim_{\Q}+\frac{(n-3)(n-2)}{2}$, if $n\geq 7$; \\
& & & & $\dim_{\Q}+1$, if $n=3$\\ \hline
\end{tabular}
\end{table}

\begin{table}
\caption{Lie groups of $\Q$-type and the corresponding Lie algebras (part 2)}\label{table_lie_algebras_Q2} 
\begin{tabular}{| c | c | c | c | c |}  \hline
Lie group  & $n\mod{4}$ & $r$ & Lie algebra & Dimension\\ \hline
$\tilde{\Q}^{\overline{01}}_{p,q,r}$, $\tilde{\Q}^{\overline{03}}_{p,q,r}$\rule{0pt}{3ex} & $2$ & $r=n-3$ & $\Lambda^{\overline{013}}_{r}\oplus\C^{\overline{2}}_{p,q,r}\oplus\{\C^2_{p,q,0}\Lambda^{n-3}_r\}$ & $\dim_{\Q}+3$ \\ 
$\tilde{\Q}^{\overline{03}}_{p,q,r}$\rule{0pt}{3ex} & $0$ & $r=n-3$ & & \\ \hline
$\tilde{\Q}^{\overline{23}}_{p,q,r}$\rule{0pt}{3ex} & $0$ & $r=n-3$ & $\Lambda^{\overline{013}}_{r}\oplus\C^{\overline{2}}_{p,q,r}\oplus\{\C^2_{p,q,0}\Lambda^{n-3}_r\}\oplus\C^{n}_{p,q,r}$  & $\dim_{\Q}+4$ \\ \hline
$\tilde{\Q}^{\overline{01}}_{p,q,r}$\rule{0pt}{3ex} & $3$& $r=n-4$ & $\Lambda^{\overline{013}}_{r}\oplus\C^{\overline{2}}_{p,q,r}\oplus\{\C^2_{p,q,0}\Lambda^{n-4}_r\}$ & $\dim_{\Q}+6$\\
$\tilde{\Q}^{\overline{03}}_{p,q,r}$\rule{0pt}{3ex} & $1$ & $r=n-4$& & \\ \hline
$\tilde{\Q}^{\overline{01}}_{p,q,r}$, $\tilde{\Q}^{\overline{03}}_{p,q,r}$\rule{0pt}{3ex} & $2$& $r=n-2$& $\Lambda^{\overline{013}}_{r}\oplus\C^{\overline{2}}_{p,q,r}\oplus\{\C^1_{p,q,0}\Lambda^{n-2}_r\}\oplus\{\C^2_{p,q,0}\Lambda^{n-3}_r\}$ & $\dim_{\Q}+n$  \\  \hline
$\tilde{\Q}^{\overline{23}}_{p,q,r}$, $\tilde{\Q}^{\overline{03}}_{p,q,r}$\rule{0pt}{3ex} & $0$& $r=n-2$& $\Lambda^{\overline{013}}_{r}\oplus\C^{\overline{2}}_{p,q,r}$ & $\dim_{\Q}+n+1$ \\  
& & & $\oplus\{\C^1_{p,q,0}\Lambda^{n-2}_r\}\oplus\{\C^2_{p,q,0}\Lambda^{n-3}_r\}\oplus\C^n_{p,q,r}$ & \\ \hline
$\tilde{\Q}^{\overline{01}}_{p,q,r}$\rule{0pt}{3ex}& $3$& $r=n-3,n-2$& $\Lambda^{\overline{013}}_{r}\oplus\C^{\overline{2}}_{p,q,r}$ & $\dim_{\Q}+\frac{(p+q)r!}{(r-n+3)!(n-3)!}$   \\
& & & $\oplus\{\C^1_{p,q,0}\Lambda^{n-3}_r\}\oplus\{\C^2_{p,q,0}\Lambda^{n-4}_r\}$ & $+\frac{(p+q)(p+q-1)r!}{2(r-n+4)!(n-4)!}$ if $n\geq7$; \\
& & & & $\dim_{\Q}+p+q$ if $n=3$ \\ \hline
$\tilde{\Q}^{\overline{03}}_{p,q,r}$\rule{0pt}{3ex} & $1$ & $r=n-3$ & $\Lambda^{\overline{013}}_{r}\oplus\C^{\overline{2}}_{p,q,r}\oplus\{\C^1_{p,q,0}\Lambda^{n-3}_r\}$ & $\dim_{\Q}+3n-3$ \\
& & & $\oplus\{\C^2_{p,q,0}\Lambda^{n-4}_r\}\oplus\{\C^2_{p,q,0}\Lambda^{n-3}_r\}$& \\ \hline
$\tilde{\Q}^{\overline{03}}_{p,q,r}$\rule{0pt}{3ex} & $1$ & $r=n-1$ &  $\Lambda^{\overline{013}}_{r}\oplus\C^{\overline{2}}_{p,q,r}$ & $\dim_{\Q}+n+\frac{(n-1)(n-2)}{2}$\\ 
& & & $\oplus\{\C^{1}_{p,q,0}\Lambda^{n-2}_r\} \oplus\{\C^{1}_{p,q,0}\Lambda^{n-3}_r\}\oplus\C^n_{p,q,r}$ & \\ \hline
$\tilde{\Q}^{\overline{03}}_{p,q,r}$\rule{0pt}{3ex} & $1$ & $r=n-2$ & $\Lambda^{\overline{013}}_{r}\oplus\C^{\overline{2}}_{p,q,r}\oplus\{\C^{1}_{p,q,0}\Lambda^{n-2}_r\}\oplus\{\C^{1}_{p,q,0}\Lambda^{n-3}_r\} $ & $\dim_{\Q}+\frac{n(n+1)}{2}$ \\ 
& & & $\oplus\{\C^{2}_{p,q,0}\Lambda^{n-3}_r\}\oplus\{\C^{2}_{p,q,0}\Lambda^{n-4}_r\}\oplus\C^n_{p,q,r}$ & \\ \hline
\end{tabular}
\end{table}
\end{landscape}

\bibliographystyle{spmpsci}
\bibliography{myBibLib}

\end{document}